\documentclass{amsart}

\usepackage{paralist}
\usepackage{url}
\usepackage{hyperref}
\usepackage{color}
\usepackage{framed}
\usepackage{mathrsfs}
\usepackage{graphicx}

\newcommand{\RR}{{\mathbb R}}

\newcommand{\N}{\mathbb{N}}

\renewcommand{\P}{\mathbb{P}}

\newcommand{\td}[1]{\tilde{#1}}

\newcommand{\co}[1]{{\mathbf{co}}\left(#1\right)}
\newcommand{\proj}[1]{{\mathbf{proj}}\left(#1\right)}
\newcommand{\wt}[1]{\widetilde{#1}}
\newcommand{\cl}[1]{\mathbf{cl}\left(#1\right)}
\newcommand{\aff}[1]{\mathbf{aff}(#1)}
\newcommand{\cal}[1]{\mathcal{#1}}
\newcommand{\tb}{\text{\upshape TH}}
\newcommand{\tbg}[1]{\wt{\tb}_{#1}(\wt{G})}
\newcommand{\og}[1]{\Omega_{#1}(\wt{G})}
\newcommand{\So}{\wt{S}^{\text{o}}}
\newcommand{\Sc}{\wt{S}^{\text{c}}}
\newcommand{\ri}[1]{\mathbf{ri}\left(#1\right)}

\newcommand{\Qk}{\mathcal{Q}_k}

\newtheorem{theorem}{Theorem}[subsection]

\newtheorem{prop}[theorem]{Proposition}

\newtheorem{lemma}[theorem]{Lemma}
\newtheorem{cor}[theorem]{Corollary}

\newtheorem{ass}[theorem]{Assumption}

\theoremstyle{definition}
\newtheorem{example}[theorem]{Example}

\newtheorem{remark}[theorem]{Remark}
\newtheorem{define}[theorem]{Definition}

\numberwithin{equation}{section}

\begin{document}
\title[Semidefinite representations of  non-compact convex sets]
{Semidefinite representations of non-compact convex sets}
\author{Feng Guo}
\address{School of Mathematical Sciences, Dalian  University of Technology,
Dalian, 116024, China}
\email{fguo@dlut.edu.cn}
\author{Chu Wang}
\address{KLMM, Academy of Mathematics and Systems Science, CAS,
Beijing, 100190, China}
\email{cwang@mmrc.iss.ac.cn}
\author{Lihong Zhi}
\address{KLMM, Academy of Mathematics and Systems Science, CAS,
Beijing, 100190, China}
\email{lzhi@mmrc.iss.ac.cn}

\begin{abstract}
We consider the problem of the  semidefinite representation of a class of
non-compact basic semialgebraic sets. We introduce the conditions of
pointedness and closedness at infinity of a semialgebraic set and show that
under these conditions our modified
hierarchies of nested theta bodies and Lasserre's relaxations
converge to the closure of the convex hull of $S$.  Moreover, if the PP-BDR
property is satisfied, our theta body and Lasserre's relaxation are
exact when the order is large enough;
if the PP-BDR property does not hold, our hierarchies convergent uniformly to
the closure of the convex hull of $S$ restricted to every fixed ball centered at
the origin.
We illustrate  through a set of examples that  the conditions of pointedness and closedness are
essential  to ensure the convergence.
Finally, we provide some strategies to  deal with cases where the conditions of
pointedness and closedness are violated.
\end{abstract}
\date{\today}
\maketitle

\noindent {\bf Key words.} Convex sets,
semidefinite representation,
theta bodies,
sums of squares, moment matrices.


\section{Introduction}

Consider the  basis semialgebraic set
\begin{equation*}
S:=\{x\in\RR^n\mid g_1(x)\ge 0, \ldots, g_m(x)\ge 0\},
\end{equation*}
where $g_i(X)\in\RR[X]:=\RR[X_1,\ldots,X_n],\ i=1,\ldots,m$. The convex hull of
$S$ is denoted by $\co{S}$ and its closure is denoted by $\cl{\co{S}}$.
Characterizing $\cl{\co{S}}$ is an important issue raised in
\cite{LecturesConOpt,LMIRS,SRCSHN,ParriloTalk}.
There is  a considerable amount of  interesting work by many people. For instance,
 using the same variables
appearing in $S$, Rostalski and Sturmfels
\cite{dualities} exploit projective varieties to explicitly find the polynomials
that describe the boundary of $\co{S}$ when $S$ is a compact real algebraic
variety;  by introducing more variables,
 theta bodies   \cite{thetabody} and  Lasserre's relaxations
\cite{convexsetLasserre} have been given to compute   $\cl{\co{S}}$
approximately or exactly when $S$ is a
compact semialgebraic set.
In this paper, we aim  to extend works in  \cite{thetabody,  convexsetLasserre}
and provide sufficient  conditions such
that the modified hierarchy of  theta bodies and Lasserre's relaxations of
non-compact semialgebraic set $S$  can still converge to the closure of the
convex hull of $S$.

Let $\td{g}_1, \ldots, \td{g}_m$ be homogenized polynomials   of  $g_1, \ldots,
g_m$ respectively.  We lift the cone of  ${S}$ to a cone of $\wt{S}^{\text{o}}
\in \RR^{n+1}$
\[
\wt{S}^{\text{o}}:=\{\td{x}\in\RR^{n+1}\mid \td{g}_1(\td{x})\ge 0,\ \ldots,\
\td{g}_m(\td{x})\ge 0,\ x_0>0\}.
\]
Let $\wt{X}:=(X_0, X_1,\ldots, X_n)$. Denote $\Qk(\wt{G})$
as the  $k$-th quadratic
module generated by
\[
\wt{G}:= \left\{\td{g}_1,\ldots,\td{g}_m,\ X_0,\ \Vert\wt{X}\Vert^2_2-1,\
1-\Vert\wt{X}\Vert^2_2\right\},
\]
and
\[
\wt{S}:=\left\{\td{x}\in\RR^{n+1}\mid \td{g}_1(\td{x})\ge 0,\ \ldots,\
\td{g}_m(\td{x})\ge 0,\ x_0\ge 0,\ \Vert \td{x}\Vert_2^2=1\right\}.
\]
Denote $\P[\wt{X}]_1:=(\RR[\wt{X}]_1\backslash\RR)\cup\{0\}$ where
$\RR[\wt{X}]_1$ is the set of linear polynomials in $\RR[\wt{X}]$.
We construct the hierarchy of theta bodies $\tbg{k}$
\[
\tbg{k}:=\left\{x\in\RR^n\mid\td{l}(1,x)\ge 0,\quad\forall\
\td{l}\in \Qk(\wt{G})\cap\P[\wt{X}]_1\right\},
\]
and Lasserre's relaxations $\og{k}$
\[
\Omega_k(\wt{G}):=\left\{
\begin{array}{c|c}
x\in\RR^n&
\begin{aligned}
&\exists y\in
\RR^{\td{s}(2k)},\ \text{s.t.}\
\mathscr{L}_y{(X_0)}=1,\\
&\mathscr{L}_y(X_i)=x_i,\ i=1,\ldots,n,\\
&M_{k-1}(X_0y)\succeq 0,\ M_{k-1}((\Vert\wt{X}\Vert_2^2-1)y)=0,\\
&M_k(y)\succeq 0,\ M_{k-k_j}(\td{g}_jy)\succeq 0,\ j=1,\ldots,m,\\
\end{aligned}
\end{array}
\right\},
\]
where
$\td{s}(k)={n+k+1\choose n+1}$ and $k_j=\lceil\deg{g_j}/2\rceil$, for every $k\in\N$.

\vspace{0.2cm}

\paragraph{\bf Our contribution:} Consider  a non-compact
basic semialgebraic set $S$.

\begin{itemize}

\item Assuming that $S$ is
closed at $\infty$ \cite{exactJacNie}  and  its homogenized cone
$\co{\cl{\wt{S}^{\text{o}}}}$ is closed and pointed
(equivalently, it contains no lines through the origin):

\begin{itemize}
\item
We  prove that the hierarchies of  $\tbg{k}$ and $\og{k}$ defined above
converge to $\cl{\co{S}}$  asymptotically.  If $\Qk(\wt{G})$ is closed, then
$\tbg{k}=\og{k}$ for $k \in \N$.

\item
If the  Putinar-Prestel's Bounded Degree Representation
(PP-BDR) \cite{convexsetLasserre} holds for $\wt{S}$ with
order $k'$, then we conclude that $\cl{\co{S}}=\tbg{k'}=\og{k'}$.
If PP-BDR property does not hold, then for every $\epsilon>0$, we show that
$\tbg{k}$ and $\og{k}$ convergent uniformly to $\cl{\co{S}}$ restricted to every
fixed ball centered at the origin.
\end{itemize}

\item
We show that the  conditions of closedness and pointedness are essential to
guarantee the convergence of the constructed hierarchies.
\begin{itemize}
\item We observe that the condition of closedness of $S$ at $\infty$ depends on
the generators of $S$ and in many cases, we can force $S$ to become closed
at $\infty$  by adding a redundant linear polynomial obtained by the
property of pointedness of $\co{\cl{\wt{S}^{\text{o}}}}$.

\item
If $\co{\cl{\wt{S}^{\text{o}}}}$ is not pointed, then we divide $S$ into $2^n$
parts along each axis.
If $S$ is closed at $\infty$ and each part satisfies PP-BDR property, we can compute
the theta bodies and Lasserre's relaxations for
each one and then glue them together
properly.
\end{itemize}
\end{itemize}

\vspace{0.2cm}
\paragraph{\bf Structure of the paper:}
We provide in Section \ref{preliminary}   some preliminaries about convex sets
and cones.  We also recall some  known results about theta bodies   \cite{thetabody}
and  Lasserre's relaxations \cite{convexsetLasserre} for compact semialgebraic
sets.  An example is given to show that  for a non-compact semialgebraic set
$S$, the sequence defined in (\ref{eq::Lasserre}) or (\ref{eq::tb}) does not
converge to $\cl{\co{S}}$.  In  Section \ref{maincontribution}, when  $S$ is a
non-compact semialgebraic set, we  provide sufficient conditions  for
guaranteeing  the convergence of modified
 Lasserre's relaxations and theta bodies for computing
 $\cl{\co{S}}$.  Some examples are also given to illustrate our method.
More discussions on these sufficient conditions are given
in Section \ref{sec::discussion}.

\section{Preliminaries}\label{preliminary}
In this section we present some  preliminaries needed in the rest of
this paper.

\subsection{Convex sets and cones} The symbol $\RR$ denotes the set of real
numbers. For $x\in\RR^n$, $\Vert x\Vert_2$ denotes the standard Euclidean norm
of $x$. A subset $C\in\RR^n$ is {\itshape convex} if for any $u, v\in C$ and any
$\theta$ with $0\le\theta\le 1$, we have $\theta u+(1-\theta)v\in C$. For any
subset $W\in\RR^n$, denote $\ri{W}$, $\cl{W}$ and $\co{W}$ as the relative
interior, closure and convex hull of $W$, respectively. A subset
$K\subseteq\RR^n$ is a {\itshape cone} if it is closed under positive scalar
multiplication. The {\itshape dual cone} of $K$ is
\[
K^*=\{c\in\RR^n\mid \langle c, x\rangle\ge 0,\quad \forall x\in K\}.
\]
In particular, $(\RR^n)^*=\RR^n$ and $L^*=L^\perp$ for any
subspace $L\in\RR^n$. A
cone $K$ need not be convex, but its dual cone $K^*$ is always convex and
closed. The second
dual $K^{**}$ is the closure of the convex hull of $K$.
Hence, if $K$ is a closed convex cone, then $K^{**}=K$.

\begin{prop}\label{prop::conedual}
Let $K_1$ and $K_2\subseteq\RR^n$ be two closed convex cone,  then
\begin{equation}\label{eq::conedual}
(K_1+K_2)^*=K_1^*\cap K_2^*\quad\text{and}\quad (K_1\cap
K_2)^*=\cl{K_1^*+K_2^*}.
\end{equation}
In particular, for any subspace $L\subseteq\RR^n$,
\[
(K_1\cap L)^*=\cl{K_1^*+L^\perp}.
\]
\end{prop}
\begin{proof}
It is clear
that $K_1^*\cap K_2^*\subseteq (K_1+K_2)^*$. To prove the first equality, it is enough  to show that  $(K_1+K_2)^* \subseteq K_1^*\cap K_2^*$.
Let $l\in (K_1+K_2)^*$, then for any $x^{(1)}\in K_1,\
x^{(2)}\in K_2,\ c_1>0,\ c_2>0$, we have $\langle l,\
c_1x^{(1)}+c_2x^{(2)}\rangle\ge 0$. Let $c_1$ and $c_2$ tend to $0$
respectively, we can get $\langle l,\ x^{(1)}\rangle\ge 0$ and
$\langle l,\ x^{(2)}\rangle\ge 0$, i.e., $l\in K_1^*\cap K_2^*$.

Since $K_1$ and $K_2$ are closed, we have $K_1^{**}=K_1$ and $K_2^{**}=K_2$. By
the first equality in (\ref{eq::conedual}), we have
\[
(K_1\cap K_2)^*=(K_1^{**}\cap K_2^{**})^*=((K_1^*+K_2^*)^*)^*=\cl{K_1^*+K_2^*}.
\]
\end{proof}
\begin{theorem}\cite[Corollary 6.5.1]{convexanalysis}\label{th::convexclosure}
Let $C$ be a convex set and let $M$ be an affine set which contains a point of
$\ri{C}$. Then
\[
\ri{M\cap C}=M\cap \ri{C},\quad \cl{M\cap C}=M\cap \cl{C}.
\]
\end{theorem}

\begin{theorem}\cite[Theorem 6.8]{convexanalysis}\label{relative point decomposation}
Let C be a convex set in $\RR^{m+p}$. For each $y\in \RR^m$, let $C_y$ be the set
of vectors $z\in \RR^p$ such that $(y,z)\in C$. Let $D=\{y\mid C_y\neq 0\}$.
Then $(y,z)\in\ri{C}$ if and only if $y\in\ri{D}$ and  $z\in \ri{C_y}$.
\end{theorem}

\begin{define}
A closed convex cone $K$ is {\itshape pointed} if $K\cap-K=\{0\}$, i.e., $K$
contains no lines through the origin.
\end{define}

\begin{prop}\cite[Section 3.3, Exercise 20]{CANO}\label{prop::pc}
Consider a closed convex cone K in $\RR^n$. A base for K is a convex set with
$0\notin\cl{C}$ and $K=\RR_+C$. The following properties are equivalent:
\begin{enumerate}[\upshape (a)]
\item $K$ is pointed;
\item $\cl{K^*-K^*}=\RR^n$;
\item $(K^*-K^*)=\RR^n$;
\item $K^*$ has nonempty interior;
\item There exists a vector $y\in\RR^n$
and real $\epsilon >0$ with $\langle
y,x\rangle\geq \epsilon \|x\|_2$, for all points $x \in K$;
\item $K$ has a bounded base.
\end{enumerate}
\end{prop}
\begin{proof}
See Appendix.
\end{proof}

It is well known that the convex hull of a compact
set in $\RR^n$ is closed.
 However, it is generally not true for a non-compact
set. For example, let
\[
V:=\{(x_1,x_2)\in\RR^2\mid \ x_1x_2=1, \,  x_2\geq0\}\cup\{(0,0)\},
\]
then
\[
\co{V}=\{(x_1,x_2)\mid x_1>0,\, x_2>0\}\cup\{(0,0)\},
\]
which is not closed.

\begin{theorem}\label{th::pointedcone}
Let  $K$ be a closed cone.
The  following assertions
are equivalent:
\begin{enumerate}[\upshape (a)]
\item $\co{K}$ contains no lines through the origin;
\item $\co{K}$ is closed  and pointed;
\item There exists a vector
$c=(c_1,\ldots, c_n)\in\RR^n$ such that $\langle c,
x\rangle>0$ for all $x\in\co{K}\backslash \{0\}$.
\end{enumerate}
\end{theorem}

\begin{proof}
By Proposition \ref{prop::pc} (e), it is sufficient to prove $(a)
\Rightarrow (b)$.

Fix a point $x\in\cl{\co{K}}$, then there is a
sequence $\{x^{(r)}\}_{r=1}^\infty\subseteq\co{K}$ such that $x^{(r)}\rightarrow
x$ as $r\rightarrow\infty$. For each $r\in\N$, by Carath\'{e}odory's Theorem,
there exist $\{x^{(r,l)}\}_{l=1}^{n+1}\subseteq K$ and
$\{\lambda_{r,l}\}_{l=1}^{n+1}\subseteq[0,1]$ such that
$\sum_{l=1}^{n+1}\lambda_{r,l}=1$ and
\begin{equation}\label{convex expression}
x^{(r)}=\sum_{l=1}^{n+1}\lambda_{r,l}x^{(r,l)}=\sum_{l=1}^{n+1}\lambda_{r,l}\Vert
x^{(r,l)}\Vert_2\frac{x^{(r,l)}}{\Vert x^{(r,l)}\Vert_2}.
\end{equation}
Since the sequence $\{x^{(r,l)}/\|x^{(r,l)}\|_2\}_{r=1}^\infty$
is bounded for each $l$, there exists a subsequence $x^{(r_t,l)}$ such that
\begin{equation}\label{eq::converge}
\lim_{t\rightarrow\infty}\frac{x^{(r_t,l)}}{\Vert x^{(r_t,l)}\Vert_2}=y_l,\quad l=1,\ldots,n+1.
\end{equation}
Because $K$ is a closed cone, each $y_l\in K$. Without loss of generality,
 we assume $(\ref{eq::converge})$ is true for the whole sequence. In the
following, we prove that the sequence $\{\lambda_{r,l}
\|x^{(r,l)}\|_2\}_{r=1}^\infty$ is bounded for each $l$.

The closed cone
$\{\sum_{l=1}^n\mu_ly_l\mid\mu_l\ge 0\}$ is pointed since it is contained in
$\co{K}$. By Proposition \ref{prop::pc} (e), there exists a unit vector $c\in\RR^n, \|c\|_2=1$ 
 and
$\epsilon>0$ such that $\langle c, y_l\rangle>\epsilon$ for each $1\le l\le
n+1$. Then there exists an
$N\in\N$ such that
\[
\frac{\langle c, x^{(r,l)}\rangle}{\Vert
x^{(r,l)}\Vert_2}>\frac{\epsilon}{2},\quad \forall r>N,\ \ 1\le l\le n+1.
\]
Therefore,
\[
\begin{aligned}
\Vert x^{(r)}\Vert_2&=\Big\Vert\sum_{l=1}^{n+1}\lambda_{r,l}
\|x^{r,l}\|_2\frac{x^{r,l}}{\|x^{(r,l)}\|_2}\Big\Vert_2\\
&=\Big\Vert\sum_{l=1}^{n+1}\frac{\lambda_{r,l}
\|x^{r,l}\|_2}{\sum_{l=1}^{n+1}\lambda_{r,l}
\|x^{r,l}\|_2}\frac{x^{r,l}}{\|x^{(r,l)}\|_2}\Big(\sum_{l=1}^{n+1}\lambda_{r,l}
\|x^{r,l}\|_2\Big)\Big\Vert_2\\
&=\Vert c\Vert_2\Big\Vert\sum_{l=1}^{n+1}\frac{\lambda_{r,l}
\|x^{r,l}\|_2}{\sum_{l=1}^{n+1}\lambda_{r,l}
\|x^{r,l}\|_2}\frac{x^{r,l}}{\|x^{(r,l)}\|_2}\Big\Vert_2\left(\sum_{l=1}^{n+1}\lambda_{r,l}
\|x^{r,l}\|_2\right)\\
&\ge\left(\sum_{l=1}^{n+1}\frac{\lambda_{r,l}
\|x^{r,l}\|_2}{\sum_{l=1}^{n+1}\lambda_{r,l}
\|x^{r,l}\|_2}\frac{\langle c, x^{r,l}\rangle}{\|x^{(r,l)}\|_2}\right)\left(\sum_{l=1}^{n+1}\lambda_{r,l}
\|x^{r,l}\|_2\right)\\
&>\frac{\epsilon}{2}\left(\sum_{l=1}^{n+1}\lambda_{r,l}
\|x^{r,l}\|_2\right).
\end{aligned}
\]
Since $x^{(r)}\rightarrow x$ as $r\rightarrow\infty$, each sequence
$\{\lambda_{r,l} \|x^{(r,l)}\|_2\}_{r=1}^\infty$ is bounded. There
exists a subsequence $\{\lambda_{r_t,l} \|x^{(r_t,l)}\|_2\}_{t=1}^\infty$ such
that
\[
\lim_{t\rightarrow\infty} \lambda_{r_t,l} \|x^{(r_t,l)}\|_2=\mu_l,\quad 1\le l\le n+1
\]
for some $\mu_l$. Without loss of generality, we assume this is true for the whole
sequence. Then
\[
\begin{aligned}
x&=\lim_{r\rightarrow\infty}x^{(r)}\\
&=\lim_{r\rightarrow\infty}\sum_{l=1}^{n+1}\lambda_{r,l}\Vert
x^{(r,l)}\Vert_2\frac{x^{(r,l)}}{\Vert x^{(r,l)}\Vert_2}\\
&=\sum_{l=1}^{n+1}\mu_ly_l \in\co{K}.
\end{aligned}
\]
Hence $\co{K}$ is closed and
$\cl{\co{K}}=\co{K}$ contains no line.
\end{proof}


\begin{remark}
Although the pointedness is defined on closed convex sets,
 by Theorem \ref{th::pointedcone},
it is safe to say that $\co{K}$ is pointed for a closed cone $K$ if  $\co{K}\cap-\co{K}=\{0\}$.
\end{remark}

\subsection{Quadratic modules and moment matrices}
 Let  $\N$ denote  the set of nonnegative integers and we set
$\N^n_k:=\{\alpha\in\N^n\mid\vert\alpha\vert=\sum_{i=1}^n \alpha_i \le k\}$ for
$k\in\N$.  The symbol $\RR[X]$ denotes the ring of multivariate polynomials in
variables $(X_1,\ldots,X_n)$ with real coefficients. For $\alpha\in\N^n$,
$X^{\alpha}$ denotes the monomial $X_1^{\alpha_1}\cdots X_n^{\alpha_n}$ whose
degree is $\vert\alpha\vert:=\sum_{i=1}^n \alpha_i$.  The symbol $\RR[X]_k$
denotes the set of real polynomials of degree at most $k$.

For any $p(X)\in\RR[X]_k$, let  $\bf{p}$ denote its column vector of
coefficients  in the monomial basis of
$\RR[X]_k$. A polynomial $p(X)\in\RR[X]$ is said to be a {\itshape sum of
squares of polynomials} (SOS) if it can be written as $p(X)=\sum_{i=1}^s
u_i(X)^2$ for some $u_1(X),\ldots,u_s(X)\in\RR[X]$.  The symbol
 $\Sigma^2$ denotes the set of polynomials that are sums of squares.

Let $G:=\{g_1,\ldots,g_m\}$ be  a set of polynomials that define the
semialgebraic set $S$.
 We denote
 \[
  \mathcal{Q}(G):=\left\{\sum_{j=0}^m\sigma_jg_j\ \Big|\ g_0=1,\
\sigma_j \in \Sigma^2 \right\}
\]
as the {\itshape quadratic module} generated by $G$ and its  $k$-th {\itshape
quadratic module}
\[
\Qk(G):=\left\{\sum_{j=0}^m\sigma_jg_j\ \Big|\ g_0=1,\
\sigma_j \in \Sigma^2, \,
\deg(\sigma_jg_j)\le 2k\right\}.
\]
It is clear that $p(x)\ge 0$ for any $p\in\mathcal{Q}(G)$ and $x\in S$.
\begin{define}\label{def::AC}
We say $\cal{Q}(G)$ satisfies the {\itshape Archimedean condition}  if
there exists $\psi\in \cal{Q}(G)$ such that the inequality $\psi(x)\ge 0$
defines a compact set in $\RR^n$.
\end{define}
Note that the  Archimedean condition implies $S$ is compact but the inverse is not
necessarily true. However, for any compact set $S$ we can always
``force'' the associated quadratic module to satisfy the condition by adding a
``redundant'' constraint $M-\Vert x\Vert^2_2$ for sufficiently large $M$.

\begin{theorem}\label{th::PP}{\upshape\cite[{\scshape Putinar's
Positivstellensatz\/}]{Putinar1993}} Suppose that $\mathcal{Q}(G)$
satisfies the Archimedean condition. If a polynomial $p\in\RR[X]$ is
positive on $S$, then $p\in\Qk(G)$ for some $k\in\N$.
\end{theorem}

\begin{define}\cite[Definition 3]{convexsetLasserre}
(Putinar-Prestel's Bounded Degree Representation of affine
polynomials) One says that {\itshape Putinar-Prestel's Bounded
Degree Representation} (PP-BDR) of affine polynomials holds for $S$
if there exists $k\in\N$ such that if $p$ is affine and positive on
$S$, then $p\in\Qk(G)$, except perhaps on a set of vectors ${\mathbf
p}\in\RR^{n}$ with Lebesgue measure zero. Call $k$ its order.
\end{define}

Let $y:=(y_{\alpha})_{\alpha \in \mathbb{N}^n_{2k}}$ be a
{\itshape truncated moment sequence} of degree $2k$.
Its  associated $k$-th {\itshape
moment matrix} is the matrix  $M_k(y)$  indexed by $\mathbb{N}^n_{k}$,
with $(\alpha,\beta)$-th entry $y_{\alpha+\beta}$ for $\alpha, \beta \in \mathbb{N}^n_{k}$.
 Given a polynomial
$p(X)=\sum_{\alpha}p_{\alpha}X^{\alpha}$,
for $k\ge d_p=\lceil \deg(p)/2\rceil$, the
$(k-d_p)$-th {\itshape localizing moment matrix} $M_{(k-d_p)}(py)$ is defined
as the moment matrix of the {\itshape shifted vector}
$((py)_{\alpha})_{\alpha\in\mathbb{N}^n_{2(k-d_p)}}$ with
$(py)_{\alpha}=\sum_{\beta}p_{\beta}y_{\alpha+\beta}$.
${\mathscr{M}_{2k}}$ denotes the space of all truncated moment sequences with degree at most   $2k$.
For any $y\in\mathscr{M}_{2k}$, the Riesz functional $\mathscr{L}_{y}$ on
$\RR[X]_{2k}$ is  defined by
\[
\mathscr{L}_y\left(\sum_{\alpha}q_{\alpha}X_1^{\alpha_1}\cdots
X_n^{\alpha_n}\right)
:=\sum_{\alpha}q_{\alpha}y_{\alpha},\quad \forall q(X)\in\RR[X]_{2k}.
\]
From the definition of the localizing moment
matrix $M_{(k-d_p)}(py)$, it is easy to check that
\begin{equation}\label{eq::M}
\mathbf{q}^TM_{(k-d_p)}(py)\mathbf{q}=\mathscr{L}_y(p(X)q(X)^2),\quad \forall
q(X)\in\RR[X]_{k-d_p}.
\end{equation}

\subsection{Lasserre's relaxations and theta bodies}
For a compact basic semialgebraic set $S\subseteq\RR^n$,
 Lasserre investigated the semidefinite representations
of $\co{S}$ in \cite{convexsetLasserre}. Let $s(k):={n+k\choose n}$ and
$k_j:=\lceil\deg{g_j}/2\rceil$ for $j=1, \ldots, m$. For any $k\in\N$, define
\begin{equation}\label{eq::Lasserre}
\Omega_k(G):=\left\{
\begin{array}{c|c}
x\in\RR^n&
\begin{aligned}
&\exists y\in\RR^{s(2k)},\ \text{s.t.}\
\mathscr{L}_y(1)=1,\\
&\mathscr{L}_y(X_i)=x_i,\ i=1,\ldots,n,\ M_k(y)\succeq 0,\\
& M_{k-k_j}(g_jy)\succeq 0,\ j=1,\ldots,m,\\
\end{aligned}
\end{array}
\right\}.
\end{equation}
It has been  proved in \cite[Theorem 2; Theorem 6]{convexsetLasserre} that
\begin{enumerate}[1.]
\item If PP-BDR property holds for $S$ with order $k$, then $\co{S}=\Omega_k(G)$;
\item Assume $\cal{Q}(G)$ is Archimedean. Then for every fixed $\epsilon>0$,
there is $k_\epsilon\in\N$ such that
$\co{S}\subseteq\Omega_{k_\epsilon}(G)\subset\co{S}+\epsilon\mathbf{B}_1$.
\end{enumerate}

Another hierarchy of semidefinite relaxations of convex hulls closely related to
$\{\Omega_k(G)\}$ is called {\itshape theta bodies} defined on real varieties
\cite{thetabody, convexAGch7}, which can be  extended to semialgebraic
sets.
 Let $\RR[X]_1$ denote the subset of all linear polynomials in $\RR[X]$,
we have
\begin{equation}\label{eq::clcoS}
\cl{\co{S}}=\bigcap_{p\in\RR[X]_1, p\vert_{S}\ge 0}\{x\in\RR^n\mid p(x)\ge 0\}.
\end{equation}
Define the $k$-th theta body of $G$ as
\begin{equation}\label{eq::tb}
\tb_k(G):=\{x\in\RR^n\mid p(x)\ge 0,\quad\forall
p\in\Qk(G)\cap\RR[X]_1\}.
\end{equation}
Clearly, we have
\[
\tb_1(G)\supseteq\tb_2(G)\supseteq\cdots\supseteq\tb_{k}(G)\supseteq\tb_{k+1}(G)\supseteq\cdots
\supseteq\cl{\co{S}}.
\]
When $\cal{Q}(G)$ is Archimedean, by Putinar's Positivstellensatz,
(\ref{eq::clcoS}) and (\ref{eq::tb}), we  have immediately
\begin{equation}\label{eq::thetabodyconverge}
\cl{\co{S}}=\bigcap_{k=1}^\infty\tb_k(G).
\end{equation}

\begin{theorem}\label{th::thetaomega}
If $\Qk(G)$ is closed, then $\tb_k(G)=\cl{\Omega_k(G)}$.
\end{theorem}

\begin{proof}
The proof is similar to the one    given in \cite[Theorem 2.8]{thetabody} for
$S$ being a real variety.
\end{proof}

The assumption of Archimedean condition plays an essential role in the hierarchy
of Lasserre's relaxations (\ref{eq::Lasserre}) and theta bodies (\ref{eq::tb}).
However,  for a non-compact semialgebraic set $S$, the Archimedean condition is
violated. We can not guarantee  that the sequence defined in
(\ref{eq::Lasserre}) or (\ref{eq::tb}) converges to $\cl{\co{S}}$. This  can be
observed from the following example.

\begin{example}\label{ex::counterex}
Consider the basic semialgebraic set
\begin{equation}\label{eq::defS}
S:=\{(x_1,x_2)\in\RR^2\mid x_1\ge 0,\
x_1^2-x_2^3\ge 0\}.
\end{equation}
As shown in Figure
\ref{fig::counterex}, $S$  defines  the gray shadow below the right half of the cusp.
Let $G:=\{X_1,\ X_1^2-X_2^3\}$. It is clear
 that the convex hull $\co{S}$ of $S$ is itself. We show a tangent line
$l(X_1,X_2):=1+2X_1-3X_2=0$ of $S$ at $(1, 1)$ in Figure \ref{fig::counterex}.

For every    $c_1X_1+c_2X_2+c_0 \in\Qk(G)\cap\RR[X]_1, c_0, c_1,
c_2\in\RR$, we have
\[
c_1X_1+c_2X_2+c_0=
\sigma_0(X_1,X_2)+\sigma_1(X_1,X_2)X_1+\sigma_2(X_1,X_2)(X_1^2-X_2^3),
\]
where $ \sigma_0, \sigma_1,\sigma_2\in \Sigma^2.$ Substituting $X_1=0$, we have
\[c_2X_2+c_0=\sigma_0(0,X_2)-X_2^3\sigma_2(0,X_2).\] Since the highest degree
terms in $\sigma_0(0,X_2)$ and $-X_2^3\sigma_2(0,X_2)$ can not cancel each other
out, we have
$\sigma_2(0,X_2)=0$ and $\sigma_0(0,X_2)$ is a constant.
This implies $c_2=0$ and
\[\tb_k(G)=\{(x_1, x_2)\in\RR^2\mid x_1\ge 0\}\]
 for all $k \in\N$.  Hence,
theta bodies $\tb_k(G)$ defined in (\ref{eq::tb}) cannot converge to $\co{S}$.
Moreover,
since $S$ has a nonempty interior, $\Qk(G)$ is closed for every $k\in\N$
\cite{PowersScheiderer, Schweighofer2005}. By Theorem \ref{th::thetaomega},
we have $\tb_k(G)=\cl{\Omega_k(G)}$ for $k\in\N$.
 Hence,
the hierarchies of Lasserre's relaxations $\Omega_k(G)$  defined in (\ref{eq::Lasserre})
cannot converge to $\co{S}$.

The main reason  that  $\tb_k(G)$
 does not converge to $\co{S}$ is  that none of
tangent lines of $S$, except $X_1=0$, can be approximated by polynomials in
$\Qk(G)\cap\RR[X]_1, \ k\in\N$. In particular,
$l_\epsilon:=l+\epsilon\not\in\Qk(G)$ for any $\epsilon>0,\ k\in\N$.
$\hfill\square$

\begin{figure}
\includegraphics[width=0.4\textwidth]{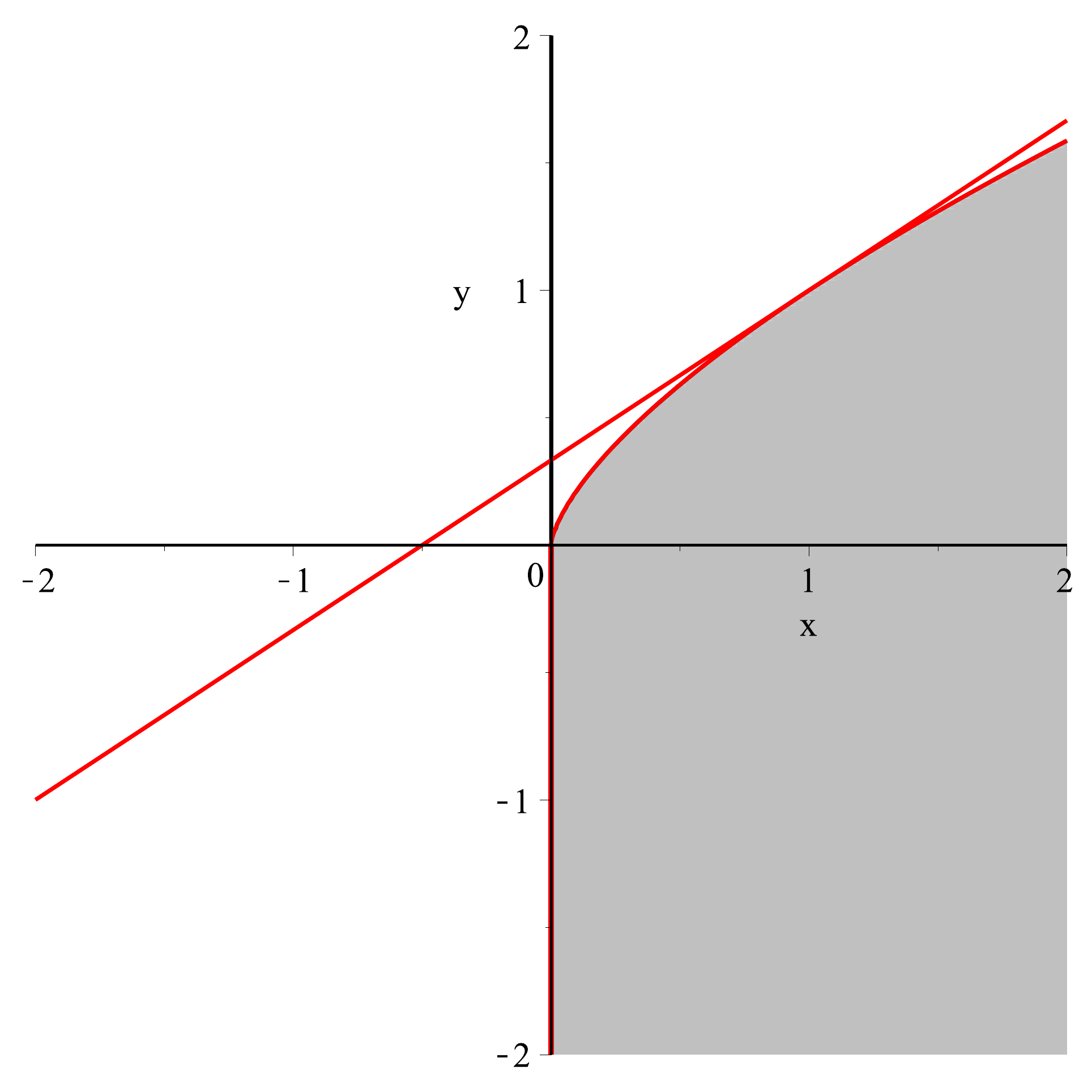}
\caption{The semialgebraic set $S$ in Example \ref{ex::counterex} and the
tangent line $l$.}\label{fig::counterex}
\end{figure}
\end{example}

\begin{remark}
Because the semialgebraic sets and projected spectrahedra in all examples in
this paper are unbounded, they are  shown in figures after being truncated
properly.
\end{remark}

In next section, we show how to overcome the difficulty in semidefinite
representations of convex hulls of non-compact semialgebraic sets.

\section{Semidefinite representations of non-compact convex sets}\label{maincontribution}
In this section, we  study how to modify   theta
bodies and Lasserre's relaxations for
 computing $\cl{\co{S}}$ when  $S$ is a   non-compact
semialgebraic set.    Our  main idea is to lift the cone of  ${S}$ to a cone of
$\wt{S}^{\text{o}}$ in $\RR^{n+1}$ via homogenization,  a technique which   has been used in  \cite{Fialkow2012,linshaowei2010} for dealing with non-compact semialgebraic sets, and  show that the
modified   theta bodies and Lasserre's relaxations
converge to $\cl{\co{S}}$ when
$S$ is closed at $\infty$ and  $\co{\cl{\wt{S}^{\text{o}}}}$ is closed and
pointed.
Some examples are  given to illustrate that the conditions of pointedness and closedness are essential
to ensure the convergence.

\subsection{Nested and closed convex approximations of $\cl{\co{S}}$}
Consider a polynomial $f(X)\in\RR[X]$ and its homogenization
$\td{f}(\wt{X})\in\RR[\wt{X}]$, where $\wt{X}=(X_0, X_1,\ldots,X_n)$ and
$\td{f}(\wt{X})=X_0^df(X/X_0),  \, d=\deg(f)$. For a given semialgebraic set
\begin{equation}\label{eq::Snew}
S:=\{x\in\RR^n\mid g_1(x)\ge 0, \ldots, g_m(x)\ge 0\},
\end{equation}
define
\begin{equation}\label{eq::Ss}
\begin{aligned}
\wt{S}^{\text{o}}&:=\{\td{x}\in\RR^{n+1}\mid \td{g}_1(\td{x})\ge 0,\ \ldots,\
\td{g}_m(\td{x})\ge 0,\ x_0>0\},\\
\wt{S}^{\text{c}}&:=\{\td{x}\in\RR^{n+1}\mid \td{g}_1(\td{x})\ge 0,\ \ldots,\
\td{g}_m(\td{x})\ge 0,\ x_0\ge 0\}, \\
\wt{S}&:=\{\td{x}\in\RR^{n+1}\mid \td{g}_1(\td{x})\ge 0,\ \ldots,\
\td{g}_m(\td{x})\ge 0,\ x_0\ge 0,\ \Vert \td{x}\Vert_2^2=1\}.
\end{aligned}
\end{equation}

\begin{prop}\cite[Proposition 2.1]{GWZ}\label{prop::eq}
$f(x)\ge 0$ on $S$ if and only if $\td{f}(\td{x})\ge 0$ on
$\cl{\wt{S}^{\text{\upshape o}}}$.
\end{prop}
\begin{cor}\label{cor::eq}
For any $f\in\RR[X]_1$, $f(x)\ge 0$ on $\cl{\co{S}}$ if and only if
$\td{f}(\td{x})\ge 0$ on $\co{\cl{\wt{S}^{\text{\upshape o}}}}$.
\end{cor}
\begin{proof}
Since $f(X)$ is linear, $f(x)\ge 0$ on $\cl{\co{S}}$ if and only if
$f(x)\ge 0$ on $S$, and $\td{f}(\td{x})\ge 0$ on $\co{\cl{\wt{S}^{\text{o}}}}$ if
and only if $\td{f}(\td{x})\ge 0$ on $\cl{\wt{S}^{\text{o}}}$. The conclusion
follows from Proposition \ref{prop::eq}.
\end{proof}

\begin{define}\cite{exactJacNie}
$S$ is closed at $\infty$ if $\cl{\wt{S}^{\text{o}}}=\wt{S}^{\text{c}}$.
\end{define}
As pointed out
in \cite[Remark 2.6]{GWZ}, not every semialgebraic set of form
(\ref{eq::Snew}) is closed at $\infty$. For instance, it is easy to prove that the set
\[
\{(x_1,x_2)\in\RR^2\mid x^2_1(x_1-x_2)-1=0,\ x_1-1\ge 0\}
\]
is not closed at $\infty$. However, it has been shown in \cite{GWZ} that the
closedness at $\infty$ is a generic property.

Let
$\P[\wt{X}]_1$ be
a  set of homogeneous polynomials of degree one in $\RR[\wt{X}]$
plus the zero polynomial.
We define
\begin{equation}\label{eq::defG}
\wt{G}:=\{\td{g}_1,\ldots,\td{g}_m,\ X_0,\ \Vert\wt{X}\Vert^2_2-1,\
1-\Vert\wt{X}\Vert^2_2\}.
\end{equation}
We consider the modified  theta bodies defined by
\begin{equation}\label{eq::newtheta}
\tbg{k}:=\{x\in\RR^n\mid\td{l}(1,x)\ge 0,\quad\forall\
\td{l}\in \Qk(\wt{G})\cap\P[\wt{X}]_1\}.
\end{equation}
Clearly, we have $\tbg{k+1}\subseteq\tbg{k}$ for each $k\in\N$.
\begin{ass}\label{assumption}
\begin{inparaenum}[\upshape(i\upshape)]
\item $S$ is closed at $\infty$;
\item The convex cone $\co{\cl{\wt{S}^{\text{\upshape o}}}}$ is
closed and pointed.
\end{inparaenum}
\end{ass}
\begin{remark}
The condition (ii) is equivalent to the other two conditions in Theorem
\ref{th::pointedcone} and can be verified by them.
\end{remark}

\begin{theorem}\label{th::main1}
Let $S\in\RR^n$ be the semialgebraic set defined as in $(\ref{eq::Snew})$.
Suppose that  Assumption \ref{assumption} is satisfied, then
$\cl{\co{S}}\subseteq\tbg{k}$ for every $k\in\N$ and
\begin{equation}\label{mainequ}
\cl{\co{S}}=\bigcap_{k=1}^\infty\wt{\tb}_k(\wt{G}).
\end{equation}
\end{theorem}
\begin{proof}

We first show $\cl{\co{S}}\subseteq\tbg{k}$ for every $k\in\N$.
 For an $\td{l}\in\Qk(\wt{G})\cap\P[\wt{X}]_1$, we have
 $\td{l}(\td{x})\ge 0$ on
$\wt{S}$. Since $\td{l}$ is homogeneous, we have $\td{l}(\td{x})\ge
0$ on $\Sc$.
Since $\So\subseteq\Sc$, we have $\td{l}(\td{x})\ge 0$ on $\co{\cl{\So}}$.
By Corollary \ref{cor::eq}, we have $\td{l}(1,x)=l(x)\ge 0$ on
$\cl{\co{S}}$, which implies that
 $\cl{\co{S}}$  is included in $\tbg{k}$ for every
$k\in\N$.  Thus, the modified theta
bodies defined in (\ref{eq::newtheta}) form a hierarchy of closed convex
approximations of  $\co{S}$ as follows:
\begin{equation}\label{eq::subs}
\tbg{1}\supseteq\tbg{2}\supseteq\cdots\supseteq\tbg{k}\supseteq\tbg{k+1}\supseteq\cdots
\supseteq\cl{\co{S}}.
\end{equation}

We now verify that this hierarchy converges to $\cl{\co{S}}$ asymptotically.
Assume  $u\notin\cl{\co{S}}$, we show that $u\notin \tbg{k}$ for some $k\in\N$.
Since
 $\cl{\co{S}}$ is closed and convex, by the hyperplane separation theorem, there exists
a vector $(f_0,\mathbf{f})\in\RR^{n+1}$ satisfies
\[
\langle\mathbf{f}, u\rangle<f_0 \quad\text{and}\quad \langle\mathbf{f},
x \rangle>f_0\quad \text{on}\quad  \cl{\co{S}}.
\]
Let  $\td{f}(\wt{X}):=\sum_{i=1}^n f_iX_i-f_0X_0\in\RR[\wt{X}]$,
then
\[
\td{f}(1,u)<0 \quad \text{and} \quad \td{f}(1,x)=f(x)>0 \quad \text{on} \quad \cl{\co{S}}.
\]
By Corollary \ref{cor::eq}, we have
\begin{equation}\label{eq::nn}
\wt{f}(\td{x})\ge 0 \quad\forall\ x\in\co{\cl{\wt{S}^{\text{o}}}}.
\end{equation}
Since $\co{\cl{\wt{S}^{\text{o}}}}$ is closed and pointed, by
Theorem \ref{th::pointedcone}, there exists a polynomial
$\td{g}(\wt{X})=\sum_{i=0}^ng_iX_i\in\P[\wt{X}]_1$ such that
$\td{g}(\td{x})>0$ on $\co{\cl{\wt{S}^{\text{o}}}}$. We choose a
small $\epsilon>0$ such that $(\td{f}+\epsilon\td{g})(1,u)<0$ and
rename $\td{f}+\epsilon\td{g}$ as $\td{f}$, then
\begin{equation}
\td{f}(1,u)<0 \quad \text{and}\quad  \td{f}(\td{x})>0 \quad
\text{on}\quad \cl{\wt{S}^{\text{o}}}.
\end{equation}

 We have assumed that $S$ is closed at $\infty$,
$\cl{\wt{S}^{\text{o}}}\cap\{\td{x}\in\RR^{n+1}\mid\Vert\td{x}\Vert_2=1\}=\wt{S},$
hence
\begin{equation}\label{existencefS}
\td{f}(1,u)<0 \quad \text{and}\quad \td{f}(\td{x})>0 \quad \text{on}
\quad  \wt{S}.
\end{equation}
 Since $\wt{S}$ is  compact, by Putinar's
Positivstellensatz, there exists a $k'\in\N$ such that $\td{f}\in
\mathcal{Q}_{k'}(\wt{G})\cap\P[\wt{X}]_1$. Since $\td{f}(1,u)<0$,
we have $u\not\in\tbg{k'}$. This implies
 \begin{equation}\label{eq::subsrev}
 \bigcap_{k=1}^\infty\wt{\tb}_{k}(\wt{G})\subseteq \cl{\co{S}}.
 \end{equation}

 Finally,  by (\ref{eq::subs}) and (\ref{eq::subsrev}),  we can
conclude $\cl{\co{S}}=\bigcap_{k=1}^\infty\wt{\tb}_{k}(\wt{G})$.
\end{proof}

\paragraph{\bf Example \ref{ex::counterex} (continued)}
By the definitions $(\ref{eq::defS})$ and $(\ref{eq::Ss})$,
\[
\begin{aligned}
\wt{S}^\text{o}&=\{(x_0, x_1, x_2)\in\RR^3\mid x_1\ge 0,\ x_0x_1^2-x_2^3\ge 0,\
x_0>0\},\\
\wt{S}^\text{c}&=\{(x_0, x_1, x_2)\in\RR^3\mid x_1\ge 0,\ x_0x_1^2-x_2^3\ge 0,\
x_0\ge 0\},\\
\wt{S}&=\{(x_0, x_1, x_2)\in\RR^3\mid x_1\ge 0,\ x_0x_1^2-x_2^3\ge 0,\ x_0\ge
0,\ \Vert \td{x}\Vert_2^2=1\}.
\end{aligned}
\]
In Figure \ref{fig::cone}, we show
the cone $\wt{S}^\text{c}$ in $\RR^3$
as well as the hyperplane
$\td{l}(X_0, X_1, X_2):=X_0+2X_1-3X_2=0$ generated by $l$.
\begin{figure}
\includegraphics[width=0.6\textwidth]{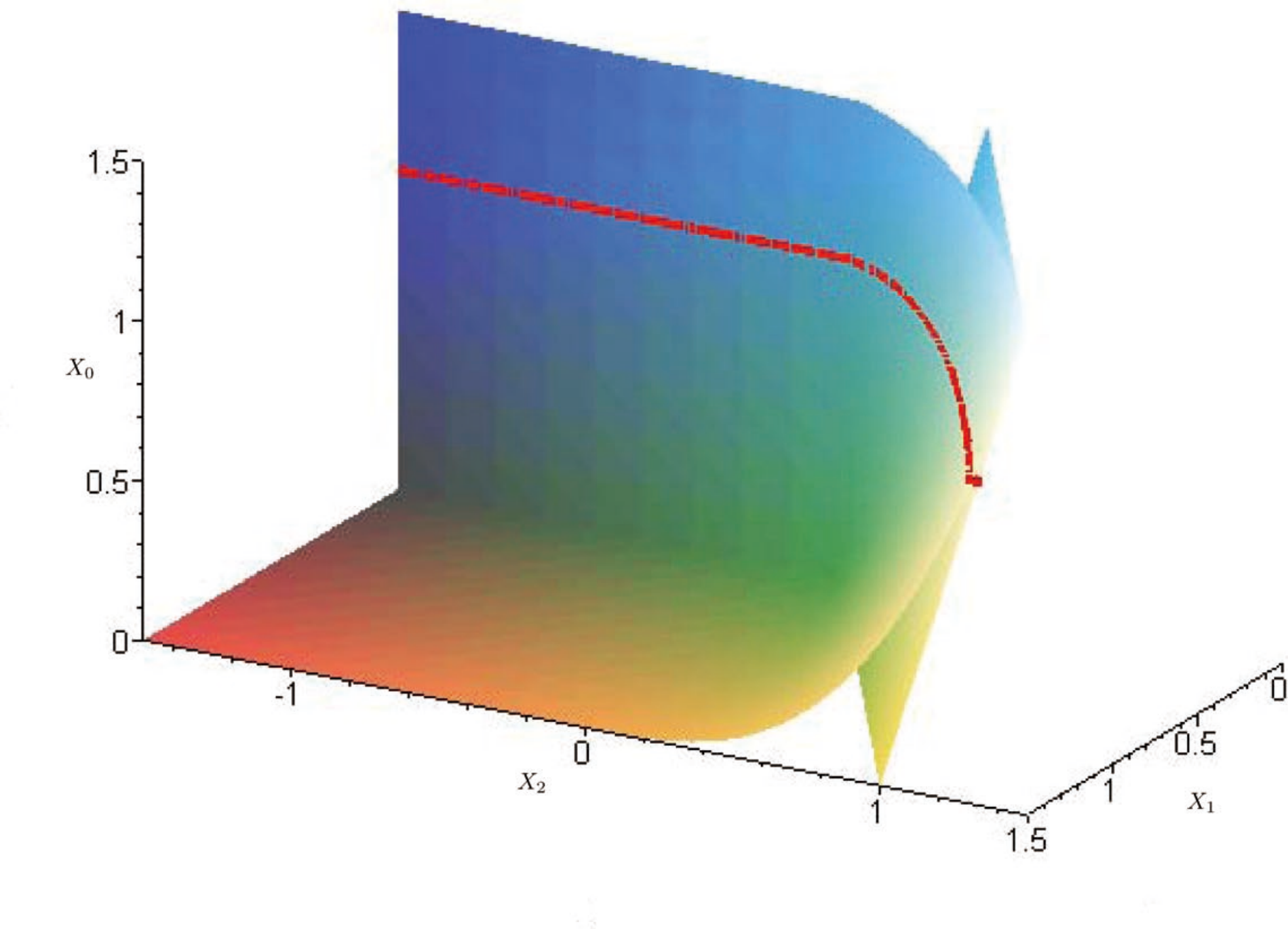}
\caption{The cone $\wt{S}^\text{c}$ and hyperplane $\td{l}$ generated by $S$ and $l$,
respectively, in Example \ref{ex::counterex}.}\label{fig::cone}
\end{figure}
It is shown in Figure \ref{fig::conesphere} that $\td{l}$ is nonnagetive on
$\wt{S}$.
\begin{figure}
\includegraphics[width=0.6\textwidth]{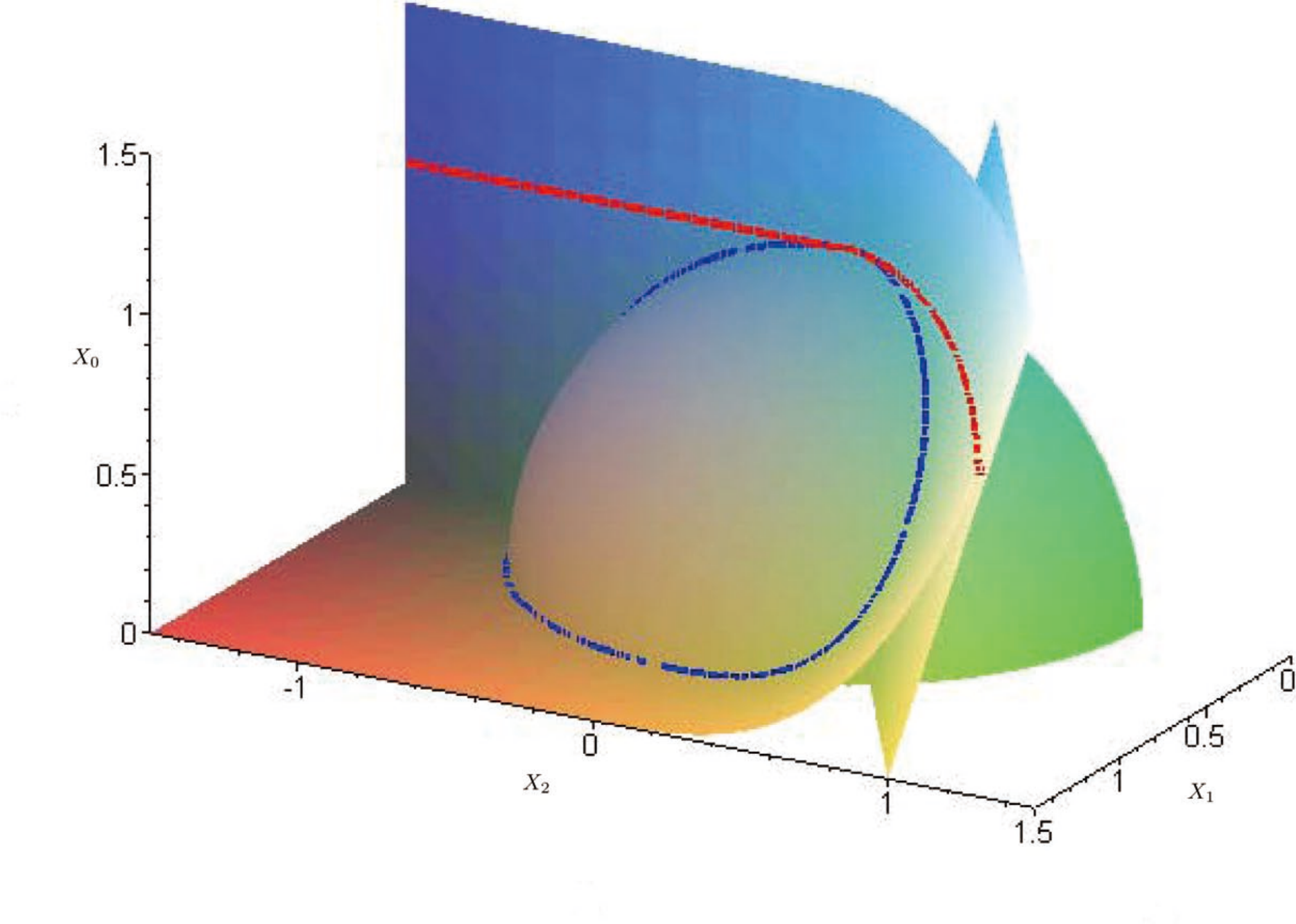}
\caption{The intersection of the cone $\wt{S}^\text{c}$ and the unit sphere in Example
\ref{ex::counterex}.}\label{fig::conesphere}
\end{figure}

For every $(0, u_1,
u_2)\in\wt{S}^{\text{c}}\backslash\wt{S}^{\text{o}}$, let
\[
u^{(\epsilon)}:=\left(\epsilon,\ u_1,\ \sqrt[3]{\epsilon u_1^2+u_2^3}\right).
\]
Then $\{u^{(\epsilon)}\}_{\epsilon>0}\subseteq\wt{S}^{\text{o}}$ and
$\lim_{\epsilon\rightarrow 0}u^{(\epsilon)}=(0, u_1, u_2)$. Hence, we have
$\wt{S}^{\text{c}}\backslash\wt{S}^{\text{o}}\subseteq\cl{\wt{S}^{\text{o}}}$ and
 $S$ is closed at $\infty$.  Moreover,
it can be verified that
\[\td{g}(X_0,X_1,X_2):=2X_0+2X_1-3X_2\] is positive on
$\co{\cl{\wt{S}^{\text{o}}}}\backslash\{0\}$ which implies
$\co{\cl{\wt{S}^{\text{o}}}}$ is pointed by Theorem \ref{th::pointedcone}. Hence
Assumption \ref{assumption} holds for $S$.

Let  $\epsilon>0$ tend to $0$, $\td{l}$ can be approximated by $\td{l}+
\epsilon \td{g}$ which are positive on $\wt{S}$.
Moreover, since
$\wt{S}$ is compact, by Putinar's Positivstellensatz,
$\td{l}+\epsilon \td{g}$ belongs to the quadratic module corresponding to
\[
\wt{G}:=\{X_0, X_1,\ X_0X_1^2-X_2^3, X_0^2+X_1^2+X_2^2-1, 1-X_0^2-X_1^2-X_2^2\}
\]
for every $\epsilon>0$.  Define
\[
\wt{\tb}_k(\wt{G}):=\{(x_1,x_2)\in\RR^2\mid\td{l}(1,x_1,x_2)\ge
0,\quad\forall\ \td{l}\in\Qk(\wt{G})\cap
\P[X_0,X_1,X_2]_1\},
\]
we have
\[
\cl{\co{S}}=\bigcap_{k=1}^\infty\wt{\tb}_k(\wt{G}).
\]

\begin{cor}\label{cor::main}
Let $S\in\RR^n$ be a semialgebraic set defined as in $(\ref{eq::Snew})$.
Suppose that Assumption \ref{assumption} is satisfied and PP-BDR property holds
for $\wt{S}$ with order $k'$, then $\cl{\co{S}}=\tbg{k'}$.
\end{cor}
\begin{proof}
Suppose that $\wt{S}$ satisfies PP-BDR property with order $k'$, for every
$\td{f}(\td{x})>0$ on $\wt{S}$, we have $\td{f}\in\mathcal{Q}_{k'}(\wt{G})$. The
inclusion $\tbg{k'}\supseteq\cl{\co{S}}$ is obvious by (\ref{eq::subs}).  Now we
verify $\cl{\co{S}} \supseteq\tbg{k'}$.

Assume that there exists a vector
$u \in \tbg{k'}$ but
$u\not\in\cl{\co{S}}$. According to (\ref{existencefS}), there
exists a linear polynomial $\td{f}\in\RR[\wt{X}]$ with $\td{f}(0)=0$
such that $\td{f}(1,u)<0$ and $\td{f}(\td{x})>0$ on $\wt{S}$. Since
$\wt{S}$ satisfies PP-BDR property with order $k'$,
we have $\td{f}\in\mathcal{Q}_{k'}(\wt{G})$. Due to the fact that
$\td{f}(1,u)<0$, we derive that  $u\not\in\tbg{k'}$. This yields the
contradiction. Thus, we have $\cl{\co{S}}=\tbg{k'}$.
\end{proof}

We would like to point out that two conditions in
Assumption \ref{assumption} can not be dropped
in Theorem \ref{th::main1}.

\begin{example}\label{ex::notclosed}
Consider the semialgebraic
set $S:=\{(x_1,x_2)\in\RR^2\mid x_2-x_1^2\ge 0\}$. Clearly,
$\cl{\co{S}}=S$. We have
\[
\begin{aligned}
\wt{S}^{\text{o}}&=\{(x_0,x_1,x_2)\in\RR^3\mid x_0x_2-x_1^2\ge 0,\ x_0>0\},\\
\wt{S}^{\text{c}}&=\{(x_0,x_1,x_2)\in\RR^3\mid x_0x_2-x_1^2\ge 0,\ x_0\ge 0\}.
\end{aligned}
\]
It is easy to check that $\cl{\wt{S}^\text{o}}$ is convex. Define
 $\td{f}(\wt{X}):=X_0+X_2$, we have $\td{f}(\td{x})>0$ on
$\co{\cl{\wt{S}^{\text{o}}}}\backslash\{0\}$
which implies
$\co{\cl{\wt{S}^{\text{o}}}}$ is closed and pointed. However,
$\wt{S}^{\text{c}}\backslash\cl{\wt{S}^{\text{o}}}=\{(0, 0, x_2)\in\RR^3\mid
x_2<0\}\neq\emptyset$ means $S$ is not closed at $\infty$.  Let
\[
\wt{G}=\{X_0,\ X_0X_2-X_1^2, X_0^2+X_1^2+X_2^2-1,\ 1-X_0^2-X_1^2-X_2^2\}.
\]
We prove that $\wt{\tb}_k(\wt{G})=\RR^2 $ for every $k\in\N$ and $\cl{\co{S}}\neq\bigcap_{k=1}^\infty\wt{\tb}_k(\wt{G})=\RR^2$.

Assume $c_0X_0+c_1X_1+c_2X_2\in\Qk(\wt{G})$, then
\begin{equation}\label{eq::1}
c_0X_0+c_1X_1+c_2X_2=\td{\sigma}_0+\td{\sigma}_1X_0+\td{\sigma}_2(X_0X_2-X_1^2)
+\td{h}(X_0^2+X_1^2+X_2^2-1),
\end{equation}
where
$\td{\sigma}_i \in \Sigma^2, i=0,1,2$ and $\td{h}\in\RR[\wt{X}]$.
 By substituting $(X_0,X_1,X_2)=(0,0,\pm 1)$ in (\ref{eq::1}), we get $\pm
c_2=\td{\sigma}_0(0,0,\pm 1)\ge 0$ which implies $c_2=0$.
Assume $c_1\neq 0$ and let
\[
x^{(1)}=\left(1, \frac{-c_0+c_1}{c_1}, \frac{(-c_0+c_1)^2}{c_1^2}\right),\quad
x^{(2)}=\left(1, -\frac{c_0+c_1}{c_1}, \frac{(c_0+c_1)^2}{c_1^2}\right).
\]
Let  $c_2=0$ and substitute  $(X_0, X_1, X_2)$ by
$x^{(1)}/\Vert x^{(1)}\Vert_2$ and $x^{(2)}/\Vert x^{(2)}\Vert_2$ in
(\ref{eq::1}) respectively,
we get
\[
\begin{aligned}
c_1&=\Vert x^{(1)}\Vert_2\left(\td{\sigma}_0\left(\frac{x^{(1)}}{\Vert
x^{(1)}\Vert_2}\right)+\td{\sigma}_1\left(\frac{x^{(1)}}{\Vert
x^{(1)}\Vert_2}\right)\frac{1}{\Vert x^{(1)}\Vert_2}\right)\ge 0,\\
-c_1&=\Vert x^{(2)}\Vert_2\left(\td{\sigma}_0\left(\frac{x^{(2)}}{\Vert
x^{(2)}\Vert_2}\right)+\td{\sigma}_1\left(\frac{x^{(2)}}{\Vert
x^{(2)}\Vert_2}\right)\frac{1}{\Vert x^{(2)}\Vert_2}\right)\ge 0,\\
\end{aligned}
\]
which implies $c_1=0$.
It contradicts the assumption $c_1\neq 0$. Hence, we have $c_1=0$.
By the definition in
(\ref{eq::newtheta}), we get $\wt{\tb}_k(\wt{G})=\RR^2$ for every
$k\in\N$. Therefore, we conclude that the assumption of closedness
 of $S$  at $\infty$  can not be dropped in Theorem \ref{th::main1}. $\hfill\square$
\end{example}

\begin{example}\label{ex::notpointed}
Consider the set $S:=\{(x_1,x_2)\in\RR^2\mid x_2^3-x_1^2\ge 0\}$. We have
$\cl{\co{S}}=\{(x_1,x_2)\in\RR^2\mid x_2\ge 0\}$ and
\[
\begin{aligned}
\wt{S}^{\text{o}}&=\{(x_0,x_1,x_2)\in\RR^3\mid x_2^3-x_0x_1^2\ge 0,\ x_0>0\},\\
\wt{S}^{\text{c}}&=\{(x_0,x_1,x_2)\in\RR^3\mid x_2^3-x_0x_1^2\ge 0,\ x_0\ge 0\}.
\end{aligned}
\]
It can be verified that  $\wt{S}^{\text{c}}\backslash\wt{S}^{\text{o}}=\{(0, x_1, x_2)\in\RR^3\mid
x_2\ge 0\}$.
Using similar arguments in  {\bf Example \ref{ex::counterex} ({continued})},
we can show that $S$ is closed at $\infty$.
However,
$\lim_{\epsilon\rightarrow
0}(\epsilon, \pm 1, \sqrt[3]{\epsilon})=(0, \pm 1, 0)$ and  $(0, \pm 1,
0)\in\cl{\wt{S}^\text{o}}$,
which implies that
 $\co{\cl{\wt{S}^{\text{o}}}}$ is not pointed.

 Let
\[
\wt{G}=\{X_0,\ X_2^3-X_0X_2^2, X_0^2+X_1^2+X_2^2-1,\ 1-X_0^2-X_1^2-X_2^2\},
\]
we show
$\wt{\tb}_k(\wt{G})=\RR^2$ for every $k\in\N$.

Assume
$c_0X_0+c_1X_1+c_2X_2\in\Qk(\wt{G})$, then
\begin{equation}\label{eq::2}
c_0X_0+c_1X_1+c_2X_2=\td{\sigma}_0+\td{\sigma}_1X_0+\td{\sigma}_2(X_2^3-X_0X_1^2)
+\td{h}(X_0^2+X_1^2+X_2^2-1),
\end{equation}
$\td{\sigma}_i \in \Sigma^2, i=0,1,2$ and $\td{h}\in\RR[\wt{X}]$.
Substituting
$(X_0,X_1,X_2)=(0, \pm 1, 0)$ in (\ref{eq::2}), we derive $c_1=0$.
Substituting
$(X_0,X_1,X_2)=(0, \pm 1, X_2)$ in (\ref{eq::2}), we have
\begin{equation}\label{eq::3}
c_2X_2=\td{\sigma}_0(0, \pm 1, X_2)+\td{\sigma}_2(0, \pm 1, X_2)X_2^3+\td{h}(0,
\pm 1, X_2)X_2^2.
\end{equation}
It is clear that $\td{\sigma}_0(0, \pm 1, X_2)$  can not have
a nonzero constant term. Hence,  the right side of the equation (\ref{eq::3}) is
divisible by $X_2^2$,  which is only possible when $c_2=0$.
By the definition of the theta body,
we derive $\wt{\tb}_k(\wt{G})=\RR^2$ for every $k\in\N$. This shows that
the assumption of pointedness of $\co{\cl{\wt{S}^{\text{o}}}}$ can not be
dropped in Theorem \ref{th::main1}. $\hfill\square$
\end{example}

Since the PP-BDR property is not generally true, similar to  Lemma 5 and Theorem 6 in
\cite[Section 2.5]{convexsetLasserre},
we
give an approximate semidefinite representation of $\cl{\co{S}}$.
For a radius $r\in\RR$, let $\mathbf{B}_r:=\{x\in\RR^n\mid\Vert
x\Vert_2\le r\}$.

\begin{lemma}\label{lem::lf}
Let  $\Omega\subset\RR^n$ be a closed convex set and let $r>0,
\epsilon>0$ be fixed.  Assume
that $(\Omega+\epsilon\mathbf{B}_1)\cap\mathbf{B}_r\neq\emptyset$ and
$u\in\mathbf{B}_r\backslash(\Omega+\epsilon\mathbf{B}_1)$, then
there exists a unit vector $\mathbf{f}\in\RR^n$ and a scalar $f^*$
with $\vert f^*\vert\le 3r+\epsilon$ such that
\begin{equation}\label{eq::lf}
\mathbf{f}^Tx\ge f^*\ \ \forall x\in\Omega\quad\text{and}\quad
\mathbf{f}^Tu<f^*-\epsilon.
\end{equation}
\end{lemma}
\begin{proof}
Since $\Omega$ is closed and convex, there is a unique projection $u^*$ of $u$
on $\Omega$. Let $\mathbf{f}:=(u^*-u)/\Vert u-u^*\Vert_2$ and
$f^*:=\mathbf{f}^Tu^*$. Using the same arguments in
the proof of \cite[Lemma 5]{convexsetLasserre}, we conclude that $(\ref{eq::lf})$.
Moreover, let
$\bar{u}\in\left(\Omega+\epsilon\mathbf{B}_1\right)\cap\mathbf{B}_r$,
 there exists $\hat{u}\in\Omega$ such that
$\Vert\bar{u}-\hat{u}\Vert_2\le\epsilon$. Hence, we have
\[
\begin{aligned}
\vert f^*\vert&\le\Vert\mathbf{f}\Vert_2\Vert u^*\Vert_2 \\
&\le\Vert u\Vert_2+\Vert u^*-u\Vert_2\\
&\le\Vert u\Vert_2+\Vert \hat{u}-u\Vert_2\\
&\le\Vert u\Vert_2+\Vert \hat{u}-\bar{u}\Vert_2+\Vert\bar{u}-u\Vert_2\\
&\le r+\epsilon+2r=3r+\epsilon
\end{aligned}
\]
\end{proof}

\begin{theorem}\label{th::main3}
Let $S\in\RR^n$ be a semialgebraic set defined as in $(\ref{eq::Snew})$.
Suppose  that Assumption \ref{assumption} holds, then for every fixed
$\epsilon>0$ and $r>0$
with $\cl{\co{S}}\cap\mathbf{B}_r\neq\emptyset$,
there exists an integer $k_{r,\epsilon}\in\N$ such that
\[
\cl{\co{S}}\cap\mathbf{B}_r\subseteq\tbg{k_{r,\epsilon}}\cap\mathbf{B}_r\subseteq
\left(\cl{\co{S}}+\epsilon\mathbf{B}_1\right)\cap\mathbf{B}_r
\] holds.
\end{theorem}
\begin{proof}
By Theorem \ref{th::main1}, we only need to prove
\begin{equation}\label{eq::uc}
\tbg{k_{r,\epsilon}}\cap\mathbf{B}_r\subseteq
\left(\cl{\co{S}}+\epsilon\mathbf{B}_1\right)\cap\mathbf{B}_r.
\end{equation}

Without loss of generality, we assume
$\left(\cl{\co{S}}+\epsilon\mathbf{B}_1\right)\cap\mathbf{B}_r\neq\mathbf{B}_r$
and let $u\in\mathbf{B}_r\backslash\left(\cl{\co{S}}+\epsilon\mathbf{B}_1\right)$
be fixed.
By Lemma \ref{lem::lf}, there exists $\mathbf{f}\in\RR^n$ with
$\Vert\mathbf{f}\Vert_2=1$ and a scalar $f^*$ with $\vert f^*\vert\le
3r+\epsilon$ such that the following condition holds
\[
\mathbf{f}^Tx\ge f^* \quad \text{on} \quad
\cl{\co{S}}\quad\text{and}\quad
\mathbf{f}^Tu<f^*-\epsilon.
\]
Define $\td{f}(\wt{X}):=\sum_{i=0}^n f_iX_i-f^*X_0\in\RR[\wt{X}]$, by Corollary
\ref{cor::eq}, $\wt{f}(\td{x})\ge 0$ on
$\co{\cl{\wt{S}^{\text{o}}}}$.
 Since $\co{\cl{\wt{S}^{\text{o}}}}$ is closed and pointed, by
Theorem \ref{th::pointedcone}, there exists a polynomial
$\td{g}(\wt{X})=\sum_{i=0}^ng_iX_i\in\P[\wt{X}]_1$ such
that
$\Vert \td{\mathbf{g}} \Vert_2=1$
 and $\td{g}(\td{x})>0$ on
$\co{\cl{\wt{S}^{\text{o}}}}$. We define a new polynomial
$$\td{p}(\wt{X}):=\frac{\epsilon}{\sqrt{1+r^2}}\td{g}(\wt{X})+\td{f}(\wt{X})\in\RR[\wt{X}],$$
then
\[
\begin{aligned}
\Vert\td{\mathbf{p}}\Vert_2&\le
\frac{\epsilon}{\sqrt{1+r^2}}\Vert\td{\mathbf{g}}\Vert_2+\Vert\wt{\mathbf{f}}\Vert_2\\
&\le \frac{\epsilon}{\sqrt{1+r^2}}+1+3r+\epsilon,
\end{aligned}
\]
and
\[
\begin{aligned}
\td{p}(1,u)&=\frac{\epsilon}{\sqrt{1+r^2}}\td{g}(1,u)+\td{f}(1,u)\\
&\le
\frac{\epsilon}{\sqrt{1+r^2}}\Vert\td{\mathbf{g}}\Vert_2\Vert(1,u)\Vert_2+\mathbf{f}^Tu-f^*\\
&<\epsilon-\epsilon=0.
\end{aligned}
\]
Let
\[
c:=\min\left\{\td{g}(\td{x})\ \Big\vert\
\td{x}\in\cl{\wt{S}^{\text{o}}},\ \Vert\td{x}\Vert_2=1\right\},
\]
then $c>0$ is well defined and $\td{p}(\td{x})\ge
c\epsilon/\sqrt{1+r^2}>0$ on
$\cl{\wt{S}^{\text{o}}}\bigcap\{\Vert\td{x}\Vert_2=1\}$.
As $S$ is closed at $\infty$, we have $\td{p}(\td{x})\ge
c\epsilon/\sqrt{1+r^2}>0$ on $\wt{S}$. Since $\wt{S}$  is compact,
by Putinar's Positivstellensatz
\cite{Putinar1993} and \cite[Theorem 6]{NieSchweighofer}, there
exists an integer $k_{r,\epsilon}\in\N$ depending only on $r$ and
$\epsilon$ such that
$\td{p}\in\mathcal{Q}_{k_{r,\epsilon}}(\wt{G})$.
From  $\td{p}(1,u) <0$, we derive
$u\not\in\tbg{k_{r,\epsilon}}$.  This implies   (\ref{eq::uc}).
\end{proof}

\subsection{Spectrahedral approximations of $\cl{\co{S}}$}
In order to fulfill computations of $\cl{\co{S}}$ via semidefinite programming,
we study an alternative description of $\wt{\tb}_k(\wt{G})$ in a dual view and
establish the connection between them. In the following, we consider moment
sequences $y$ of real numbers indexed by $(n+1)$-tuple $\alpha:=(\alpha_0,
\alpha_1,\ldots,\alpha_n)\in\N^{n+1}$. Each $y$ defines a Riesz functional
$\mathscr{L}_y$ on $\RR[\wt{X}]$. Recall that
$\td{s}(k)={n+k+1\choose n+1}$ and $k_j=\lceil\deg{g_j}/2\rceil$. For every
$k\in\N$, define
\begin{equation}\label{eq::omega}
\Omega_k(\wt{G}):=\left\{
\begin{array}{c|c}
x\in\RR^n&
\begin{aligned}
&\exists y\in\RR^{\td{s}(2k)},\ \text{s.t.}\
\mathscr{L}_y{(X_0)}=1,\\
&\mathscr{L}_y(X_i)=x_i,\ i=1,\ldots,n,\\
&M_{k-1}(X_0y)\succeq 0,\ M_{k-1}((\Vert\wt{X}\Vert_2^2-1)y)=0,\\
&M_k(y)\succeq 0,\ M_{k-k_j}(\td{g}_jy)\succeq 0,\ j=1,\ldots,m\\
\end{aligned}
\end{array}
\right\}.
\end{equation}

\begin{theorem}\label{th::omegatheta}
We have $\cl{\co{S}}\subseteq\cl{\og{k}}\subseteq\tbg{k}$ for every $k\in\N$.
\end{theorem}
\begin{proof}
Since $\tbg{k}$ is closed, it is sufficient to prove
\[
\co{S}\subseteq\og{k}\subseteq\tbg{k}\quad\text{for each } k\in\N.
\]
Fixing a vector
$u\in S$, let $\td{u}:=(1, u)/\Vert(1, u)\Vert_2\in\wt{S}$ and
$y:=\{y_{\alpha}\}_{\vert\alpha\vert\le 2k}$, where
$y_{\alpha}=\td{u}^{\alpha}/\td{u}_0$ for $\alpha\in\N^{n+1}$.
 It is quite  straightforward to show that  $\mathscr{L}_y{(X_0)}=1$,
$\mathscr{L}_y(X_i)=u_i,\ i=1,\ldots,n$ and
\[
\begin{aligned}
&\left\langle\mathbf{w}, M_{k}(y)\mathbf{w}\right\rangle=\frac{1}{\td{u}_0}
w^2(\td{u})\ge 0,\ \forall w\in\RR[\wt{X}]_{k},\\
&\left\langle\mathbf{v}, M_{k-1}(X_0y)\mathbf{v}\right\rangle=\frac{1}{\td{u}_0}
\td{u}_0v^2(\td{u})=v^2(\td{u})\ge 0,\ \forall v\in\RR[\wt{X}]_{k-1},\\
&\left\langle\mathbf{p},
M_{k-1}((\Vert\wt{X}\Vert_2^2-1)y)\mathbf{p}\right\rangle=\frac{1}{\td{u}_0}
p^2(\td{u})\left(\Vert\td{u}\Vert_2^2-1\right)=0,\ \forall p\in\RR[\wt{X}]_{k-1},\\
&\left\langle\mathbf{q},
M_{k-k_j}(\td{g}_jy)\mathbf{q}\right\rangle=\frac{1}{\td{u}_0}
q^2(\td{u})\td{g}_j(\td{u})\ge 0,\ \forall q\in\RR[\wt{X}]_{k-k_j},\ j=1,\ldots,m.
\end{aligned}
\]
Therefore, we derive $u\in\og{k}$ and $S\subseteq\og{k}$.
Since $\og{k}$ is convex, it is clear that  we have   $\co{S}\subseteq\og{k}$.

For  a given $v\in\og{k}$, let $y\in\RR^{\td{s}(2k)}$ be its associated
moment sequence defined in (\ref{eq::omega}).  For every
$\td{l} \in \Qk(\wt{G})\cap\P[\wt{X}]_1$, we have the representation
\[
\td{l}(\wt{X})=\td{\sigma}+\td{\sigma}_0X_0+\td{h}(\Vert\wt{X}\Vert_2^2-1)+
\sum_{j=1}^m\td{\sigma}_j\td{g}_j,
\]
where $\td{\sigma}_j$'s are SOS and each term in the summation has degree $\le 2k$.
We have
\[
\td{l}(1,
v)=\mathscr{L}_y(\td{l})=\mathscr{L}_y\left(\td{\sigma}+\td{\sigma}_0X_0+
\td{h}(\Vert\wt{X}\Vert_2^2-1)+\sum_{j=1}^m\td{\sigma}_j\td{g}_j\right)
\]
By $(\ref{eq::M})$ and $(\ref{eq::omega})$, we obtain $\td{l}(1, v)\ge 0$, which implies
$v\in\tbg{k}$. The proof is completed.
\end{proof}

\paragraph{\bf Example \ref{ex::counterex} (continued)}
 Using the software package Bermeja \cite{Bermeja}, we draw the
third order spectrahedron $\Omega_3(\wt{G})$. As  shown in Figure \ref{fig::counteroemga},
our modified Lasserre's relaxation
$\Omega_3(\wt{G})$  is a very tight approximation of
$\cl{\co{S}}$. $\hfill\square$
\begin{figure}
\includegraphics[width=0.4\textwidth]{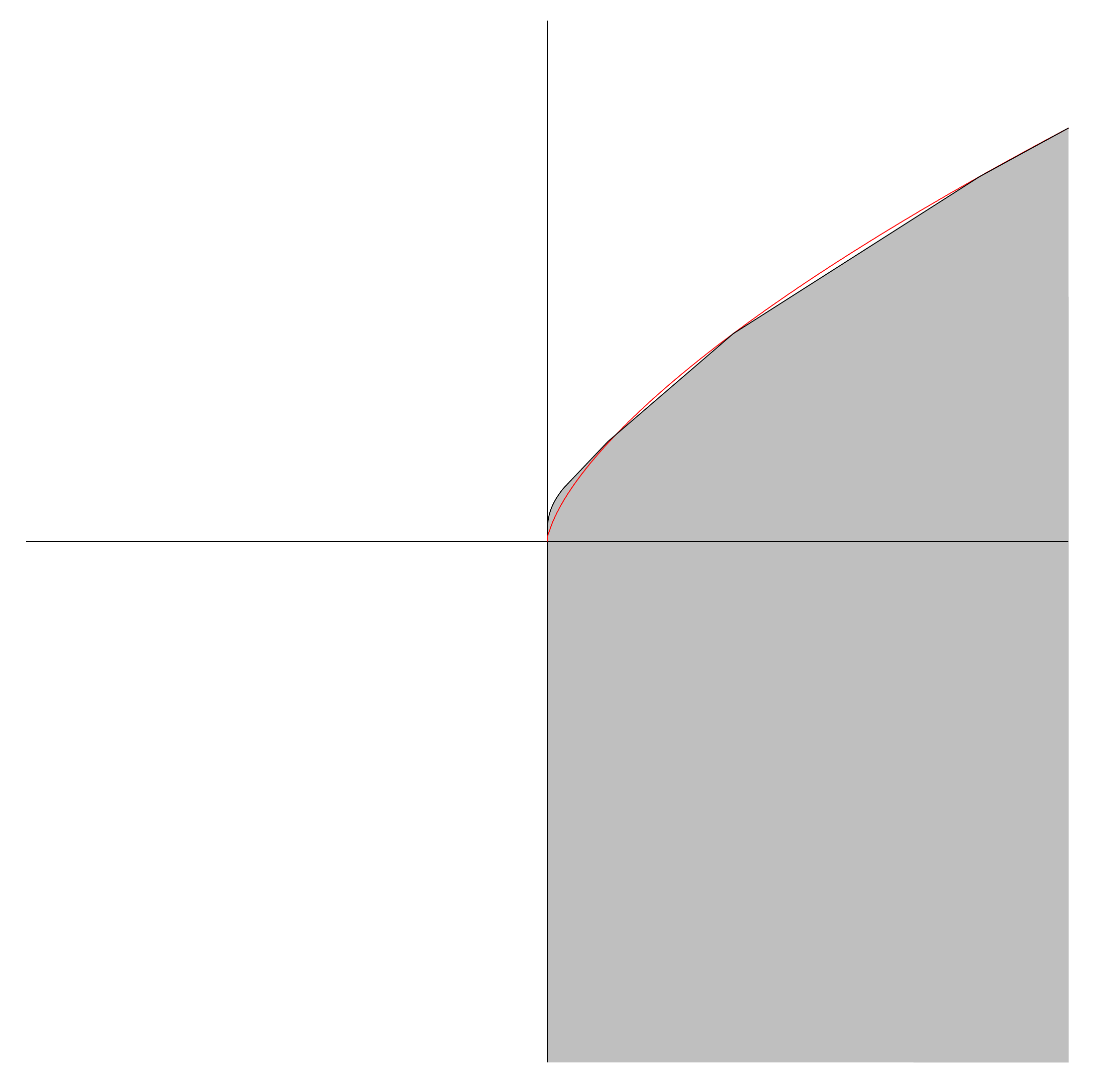}
\caption{The spectrahedral approximation $\Omega_3(\wt{G})$ (shown shaded) of
$\cl{\co{S}}$ in Example \ref{ex::counterex}.}\label{fig::counteroemga}
\end{figure}

The following results are obtained by replacing  $\tbg{k}$ by $\cl{\og{k}}$ in
Theorem \ref{th::main1}, Corollary \ref{cor::main} and Theorem \ref{th::main3}.
\begin{cor}\label{cor::mainomega}
Let $S\in\RR^n$ be a semialgebraic set defined as in $(\ref{eq::Snew})$.
Suppose that Assumption \ref{assumption} is satisfied, then
\begin{enumerate}[1.]
\item $\cl{\co{S}}\subseteq\og{k}$ for every $k\in\N$ and
$\cl{\co{S}}=\bigcap_{k=1}^\infty\cl{\og{k}}$.
\item If PP-BDR property holds for $\wt{S}$ with order $k'$, then
$\cl{\co{S}}=\cl{\og{k'}}$.
\item For every fixed $\epsilon>0$ and $r>0$ with
$\cl{\co{S}}\cap\mathbf{B}_r\neq\emptyset$, there exists an integer
$k_{r,\epsilon}\in\N$ such that
\[
\cl{\co{S}}\cap\mathbf{B}_r\subseteq\cl{\og{k_{r,\epsilon}}}\cap\mathbf{B}_r\subseteq
\left(\cl{\co{S}}+\epsilon\mathbf{B}_1\right)\cap\mathbf{B}_r
\] holds.
\end{enumerate}
\end{cor}
\begin{proof}
 It is straightforward to verify these results via the preceding theorem.
\end{proof}

\begin{example}\label{ex::ex2}
Consider the following semialgebraic set
\[
S:=\{(x_1,x_2)\in\RR^2\mid x_1^3-x_2^2-x_1+1=0,\ x_2\ge 0\},
\]
which is the red curve shown in Figure \ref{fig::ex2}.  We have
\[
\begin{aligned}
\wt{S}^\text{o}&=\{(x_0,x_1,x_2)\in\RR^3\mid x_1^3-x_0x_2^2-x_0^2x_1+x_0^3=0,\
x_2\ge 0,\ x_0> 0\},\\
\wt{S}^\text{c}&=\{(x_0,x_1,x_2)\in\RR^3\mid x_1^3-x_0x_2^2-x_0^2x_1+x_0^3=0,\
x_2\ge 0,\ x_0\ge 0\}.
\end{aligned}
\]
Clearly, $\wt{S}^\text{c}\backslash{\wt{S}^\text{o}}
=\{(0,0,x_2)\in\RR^3\mid x_2\ge 0\}$.
It can be
verified that
$\wt{S}^\text{c}=\cl{\wt{S}^\text{o}}$,
i.e., $S$ is closed at $\infty$. Since $X_0+X_2>0$ on
$\cl{\wt{S}^\text{o}}\backslash\{0\}$, we have $X_0+X_2>0$ on
$\co{\cl{\wt{S}^\text{o}}}\backslash\{0\}$ and thus,
by Theorem \ref{th::pointedcone}, $\co{\cl{\wt{S}^\text{o}}}$
is closed and pointed.
Hence, Assumption \ref{assumption} holds for $S$. The third projected
spectrahedron $\og{3}$ is depicted (shaded) in Figure \ref{fig::ex2}.
$\hfill\square$
\begin{figure}
\includegraphics[width=0.4\textwidth]{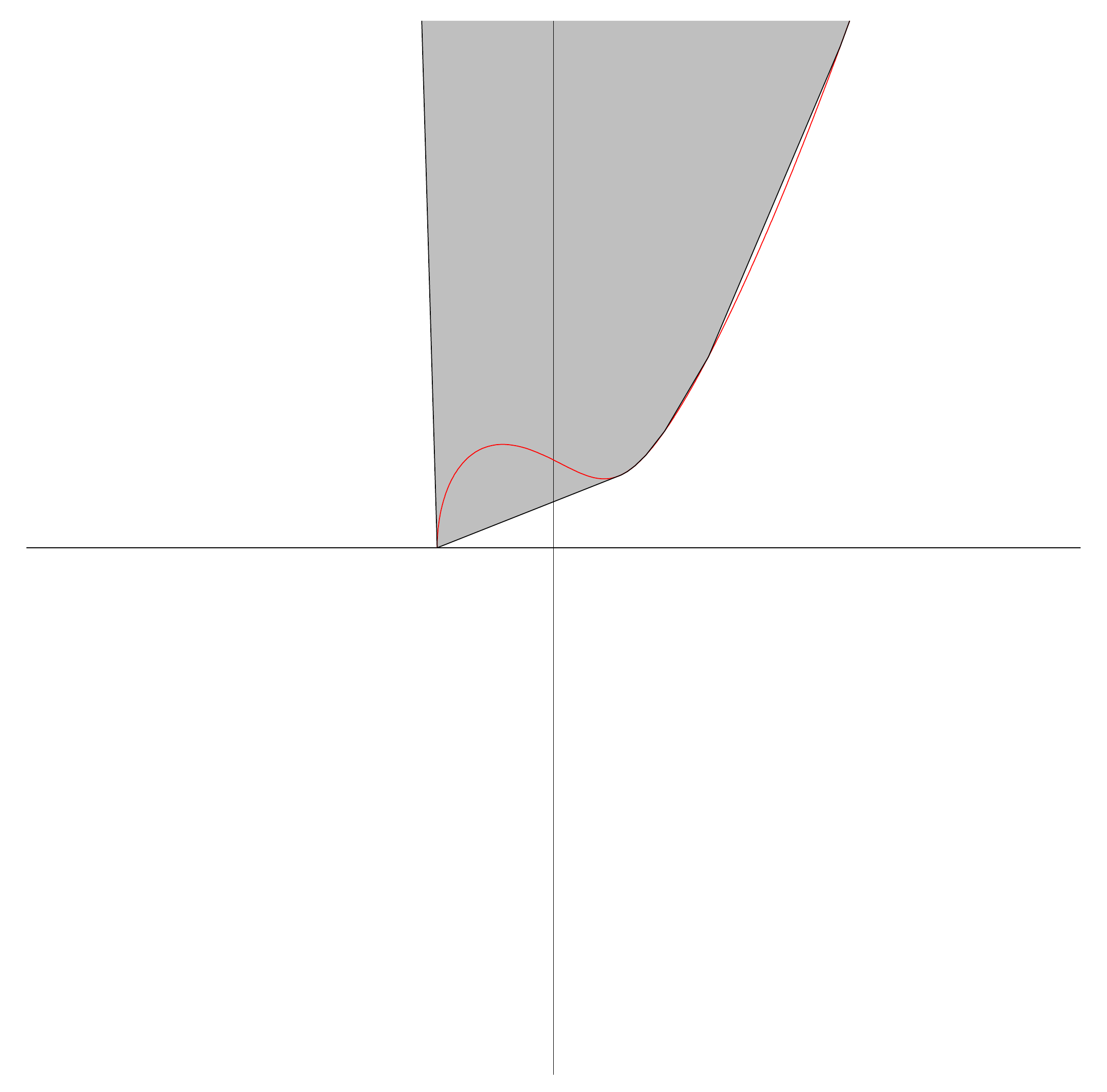}
\caption{The semialgebraic set $S$ (red curve) and projected spectrahedron $\og{3}$
 (shaded) in Example \ref{ex::ex2}.}\label{fig::ex2}
\end{figure}
\end{example}

\begin{remark}
 As shown in Theorem \ref{th::main3} and Corollary \ref{cor::mainomega}, the
projected spectrahedra $\og{k}$ for ${k\in\N}$ are outer approximations of
$\cl{\co{S}}$ and convergent uniformly to $\cl{\co{S}}$ restricted to every
fixed ball $\mathbf{B}_r$. If we truncated $S$ first by the ball $\mathbf{B}_r$
and then compute the convex hull of the resulting compact set by Lasserre's
relaxations (\ref{eq::Lasserre}), in general, we can not get approximations of
the truncation $\cl{\co{S}}\cap \mathbf{B}_r$.
Taking Example \ref{ex::ex2} for instance, compared with Figure \ref{fig::ex2},
Lasserre's relaxation of $\cl{\co{S\cap{\mathbf{B}_r}}}$ shown in Figure
\ref{fig::ex2Lasserre}  is not an outer
approximation of the truncation
$\cl{\co{S}}\cap\mathbf{B}_r$.
\begin{figure}
\includegraphics[width=0.4\textwidth]{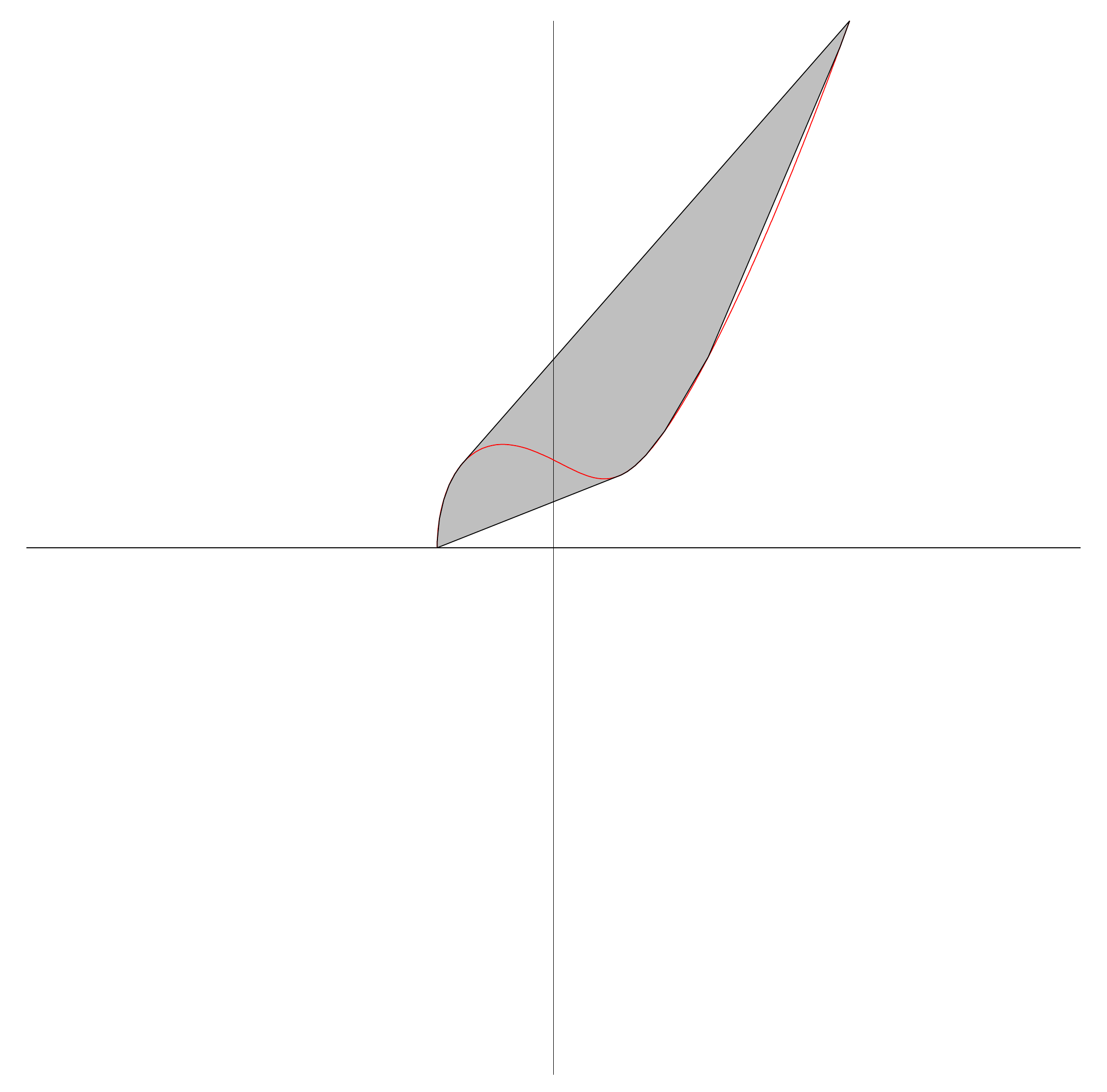}
\caption{Lasserre's relaxation of
$\cl{\co{S}}\cap{\mathbf{B}_r}$ in Example
\ref{ex::ex2}.}\label{fig::ex2Lasserre}
\end{figure}
\end{remark}

 By similar arguments given in \cite{thetabody, convexAGch7}, we
show below  that  $\tbg{k}=\cl{\og{k}}$ for each $k\in\N$  if $\Qk(\wt{G})$ is closed.

Let $M:=\{(1, x)\mid x\in\RR^{n}\}$ and
$\Qk^1(\wt{G})=\Qk(\wt{G})\cap\P[\wt{X}]_1$. By the definition of dual cones,
\begin{equation}\label{eq::qt}
\Qk^1(\wt{G})^*\cap M=\{1\}\times \tbg{k}.
\end{equation}
Denote $\proj{\Qk(\wt{G})^*}$ as the projection of $\Qk(\wt{G})^*$ onto
$\left(\P[\wt{X}]_1\right)^*$.
It is clear that
\begin{equation}\label{eq::po}
\proj{\Qk(\wt{G})^*}\cap M=\{1\}\times \Omega_k(\wt{G}).
\end{equation}
If $\Qk(\wt{G})$ is closed, by Proposition \ref{prop::conedual}, we
have
\begin{equation}\label{eq::proj}
\Qk^1(\wt{G})^*=\cl{\proj{\Qk(\wt{G})^*}}.
\end{equation}
\begin{lemma}\label{lem::ri}
If $\Qk(\wt{G})$ is closed, the hyperplane $M$ intersects
$\ri{\proj{\Qk(\wt{G})^*}}$.
\end{lemma}
\begin{proof}
By \cite[Theorem 6.3]{convexanalysis} and $(\ref{eq::proj})$, it is
equivalent to prove $M$ intersects $\ri{\Qk^1(\wt{G})^*}$. Fixing a vector
$u\in S$, then we have $\td{l}:=(1,u)/\Vert(1,u)\Vert_2\in\wt{S}$ and
 $\td{l}\in\Qk^1(\wt{G})^*$. Let
\[D:=\{t_0\in\RR\mid\exists\ t\in\RR^n,\ \text{s.t. } (t_0,
t)\in\Qk^1(\wt{G})^*\}.\]
 Since $X_0\in\Qk^1(\wt{G})$ and
$c\cdot\td{l}\in\Qk^1(\wt{G})^*$ for all $c\ge 0$, we get $D=[0,\infty)$
and thus $1\in\ri{D}$.  By  Theorem \ref{relative point
decomposation}, we derive  that $M$ intersects $\ri{\Qk^1(\wt{G})^*}$.
\end{proof}
\begin{theorem}\label{th::eq}
If $\Qk(\wt{G})$ is closed, then
$\wt{\tb}_k(\wt{G})=\cl{\Omega_k(\wt{G})}$.
\end{theorem}
\begin{proof}
By (\ref{eq::qt}), (\ref{eq::po}), (\ref{eq::proj}), Theorem
\ref{th::convexclosure} and Lemma \ref{lem::ri}, we have
\[
\begin{aligned}
\{1\}\times\cl{\Omega_k(\wt{G})}&=\cl{\proj{\Qk(\wt{G})^*}\cap M}\\
&=\cl{\proj{\Qk(\wt{G})^*}}\cap M\\
&=\Qk^1(\wt{G})^*\cap M\\
&=\{1\}\times\tbg{k}.
\end{aligned}
\]
This shows that $\wt{\tb}_k(\wt{G})=\cl{\Omega_k(\wt{G})}$.
\end{proof}

\section{More discussions on Assumption \ref{assumption}}\label{sec::discussion}

As we have seen, if Assumption \ref{assumption} is satisfied, then we can obtain a
hierarchy of nested semidefinite relaxations converging to $\cl{\co{S}}$.
In this section, we give more discussions on cases where
 Assumption
\ref{assumption} does not hold.

\subsection{Closedness at $\infty$ of $S$} We have mentioned that
a  semialgebraic set  is closed at $\infty$ in general \cite{GWZ}.
Unfortunately, as we show below, the closedness condition does not hold on
certain kinds of semialgebraic sets.

 Let $U$ be a semialgebraic set defined as
\begin{equation}\label{eq::defU}
U:=\left\{x\in\RR^n\ \Big|\
\begin{aligned}
&g_i(x)=0,\ g_j(x)\ge 0,\ i=1,\ldots,m_1,\\
&\deg(g_j)\ \text{is even, } j=1,\ldots, m_2
\end{aligned}\right\}.
\end{equation}
Denote
\[
\begin{aligned}
\wt{U}^\text{o}&=\left\{\td{x}\in\RR^{n+1}\mid\td{g}_i(\td{x})=0,\ \td{g}_j(\td{x})\ge 0,\
x_0>0,\ \ i=1,\ldots,m_1,\ j=1,\ldots, m_2\right\},\\
\wt{U}^\text{c}&=\left\{\td{x}\in\RR^{n+1}\mid\td{g}_i(\td{x})=0,\ \td{g}_j(\td{x})\ge 0,\
x_0\ge 0,\ \ i=1,\ldots,m_1,\ j=1,\ldots, m_2\right\}.
\end{aligned}
\]
\begin{prop}\label{prop::U}
Suppose $U$ is not compact. If $\co{\cl{\wt{U}^\text{\upshape o}}}$ is closed
and pointed, then $U$ is not closed at $\infty$.
\end{prop}
\begin{proof}
Since $U$ is not compact, there is a sequence $\{u^{(k)}\}_{k=1}^\infty\subseteq U$ satisfying 
$\lim_{k\rightarrow\infty}\Vert u^{(k)}\Vert_2=\infty.$
Because $\{(1,u^{(k)})/\Vert(1,u^{(k)})\Vert_2\}\subseteq\wt{U}^\text{o}$ is
bounded, there exists a nonzero point
$\td{u}=(0,u_1,\ldots,u_n)\in\cl{\wt{U}^\text{o}}$.

If
$\co{\cl{\wt{U}^\text{o}}}$ is closed and pointed,
by Theorem \ref{th::pointedcone}, we have $-\td{u}\not\in\cl{\wt{U}^\text{o}}$.
However,
 $\deg(g_j)$ is   even for $ j=1,\ldots, m_2$,
it is straightforward  to see both
$\td{u}$ and $-\td{u}$ belong to $\wt{U}^\text{c}$,
which implies $\cl{\wt{U}^\text{o}} \neq \wt{U}^\text{c}$. Therefore,
$U$ is not closed at $\infty$.
\end{proof}

\begin{remark}\label{re::closed}
Consider the semialgebraic set $S$ defined as in
(\ref{eq::Snew}). Let  $\hat{g}_i$ be the
homogeneous part of the highest degree of $g_i$ for $i=1, \ldots, m$ and
 $D_S=\wt{S}^\text{c}\backslash\cl{\wt{S}^\text{o}}$. If $S$ is not closed at
$\infty$, then
\[
\emptyset\neq D_S\subseteq\{(0, x)\in\RR^{n+1}\mid \hat{g}_1(x)\ge 0,\ \ldots,\
\hat{g}_m(x)\ge 0\}.
\]
Decompose $D_S=D_S^1\cup D_S^2$ where
\[
D_S^1=\left\{(0, x)\in\RR^{n+1}\ \Big|\ (0,
x)\in\wt{S}^\text{c}\backslash\cl{\wt{S}^\text{o}}\ \text{but }
(0,-x)\in\cl{\wt{S}^\text{o}}\right\}
\]
and
\[
D_S^2=\left\{(0, x)\in\RR^{n+1}\ \Big|\ (0,
x)\in\wt{S}^\text{c}\backslash\cl{\wt{S}^\text{o}}\ \text{and }
(0,-x)\not\in\cl{\wt{S}^\text{o}}\right\}.
\]
If $S$ is defined by (\ref{eq::defU}), then by the proof of Proposition
\ref{prop::U}, $D_S^1\neq\emptyset$.  However,  if $\co{\cl{\wt{S}^\text{o}}}$
is closed and pointed,  there exists a linear function
$\td{l}\in\P[\wt{X}]_1$ such that $\td{l}(\td{x})>0$ on
$\co{\cl{\wt{S}^\text{o}}}\backslash\{0\}$.  Adding  the inequality
$\td{l}(\td{x})\ge 0$ to the generators  of $\wt{S}^\text{o}$, or equivalently,
adding $\td{l}(1,x)\ge 0$ to the generators of $S$,  it is clear
that
$S,\wt{S}^\text{o}$ remain the same but we have $D_S^1=\emptyset$. As a result,
the set $S$ with new generators is more likely to be closed at $\infty$. This
shows that the closedness at $\infty$ depends not only on the geometry of $S$,
but also on the generators of $S$.
\end{remark}

\paragraph{\bf Example \ref{ex::notclosed} (continued)}
We have shown that the semialgebraic set $S$ is not closed at $\infty$, which
can also be verified by Proposition \ref{prop::U}.  According to  Remark
\ref{re::closed}, we  add
$x_0+x_2\ge 0$
to the  generators  of
$\wt{S}^\text{o}$, or equivalently, we add
$1+x_2\ge 0$
to the generators of
$S$, then $\wt{S}^{\text{c}}\backslash\cl{\wt{S}^{\text{o}}}=\emptyset$.
Therefore,  $S$ is closed at $\infty$ with respect to its new generators.  Since
the geometry of $S$  does not change, $\co{\cl{\wt{S}^{\text{o}}}}$ is still
closed and pointed. The second order spectrahedron $\Omega_2(\wt{G})$ is shown
(shaded) in Figure \ref{fig::notclosed}. $\hfill\square$
\begin{figure}
\includegraphics[width=0.4\textwidth]{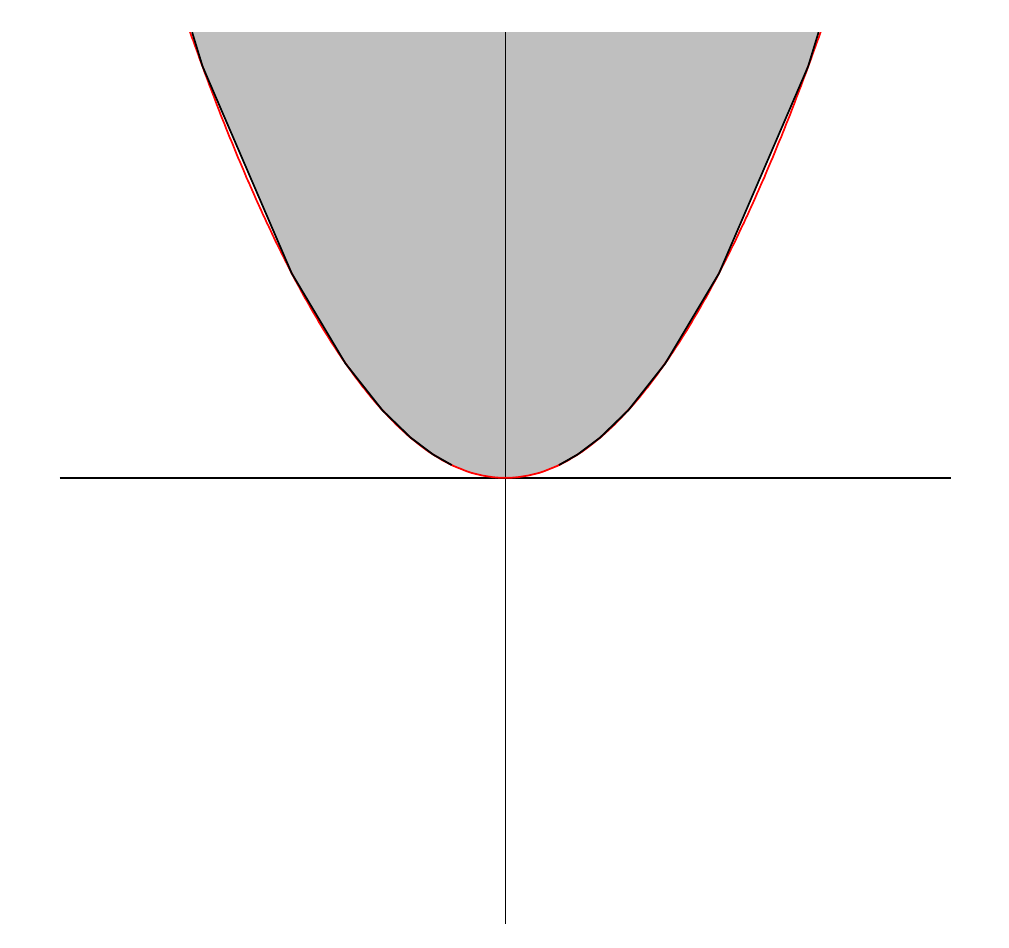}
\caption{The parabola $X_2-X_1^2=0$ (red curve) and the second order
spectrahedron $\Omega_2(\wt{G})$ (shaded) in Example
\ref{ex::notclosed}.}\label{fig::notclosed}
\end{figure}

\begin{example}\label{ex::notclosed2}
Consider the quartic bow curve
\[
S:=\{(x_1, x_2)\in\RR^2\mid x_1^4-x_1^2x_2+x_2^3=0\}
\]
as shown (red) in Figure \ref{fig::notclosed2}. We have
\[
\begin{aligned}
\wt{S}^{\text{o}}&=\{(x_0,x_1,x_2)\in\RR^3\mid x_1^4-x_0x_1^2x_2+x_0x_2^3= 0,\ x_0>0\},\\
\wt{S}^{\text{c}}&=\{(x_0,x_1,x_2)\in\RR^3\mid x_1^4-x_0x_1^2x_2+x_0x_2^3= 0,\ x_0\ge 0\}.
\end{aligned}
\]
We first show that $\co{\cl{\wt{S}^{\text{o}}}}$ is closed and pointed by
proving the polynomial $X_0-X_2$ is positive on
$\co{\cl{\wt{S}^{\text{o}}}}\backslash\{0\}$.
 For every $0\neq\td{u}=(u_0, u_1,
u_2)\in\cl{\wt{S}^{\text{o}}}$, we have $u_1^4-u_0u_1^2u_2+u_0u_2^3= 0$. If
$u_2=0$, then $u_1=0$ and $u_0-u_2>0$. Assume $u_2\neq 0$, then
$u_1^2u_2-u_2^3\neq 0$. Otherwise, we have $u_1^2=u_2^2$ and
$u_2^4-u_0u_2^3+u_0u_2^3=0$, then $u_2=0$, a contradiction. Therefore,
\[
\begin{aligned}
u_0-u_2&=\frac{u_1^4}{u_1^2u_2-u_2^3}-u_2\\
&=\frac{(u_1^2-u_2^2)^2+u_1^2u_2^2}{u_1^2u_2-u_2^3}\\
&> 0,
\end{aligned}
\]
which implies
 $\co{\cl{\wt{S}^{\text{o}}}}$ is closed and pointed. Since $S$ is
of form (\ref{eq::defU}), by Proposition \ref{prop::U}, $S$ is not
closed at $\infty$. By Remark \ref{re::closed}, we
add
$1-x_2\ge0$ to the
generators of $S$ to ``force'' it to be closed at
$\infty$. The third order spectrahedron $\Omega_3(\wt{G})$ with the new
generating set is shown (shaded) in Figure \ref{fig::notclosed2}.
$\hfill\square$
\begin{figure}
\includegraphics[width=0.4\textwidth]{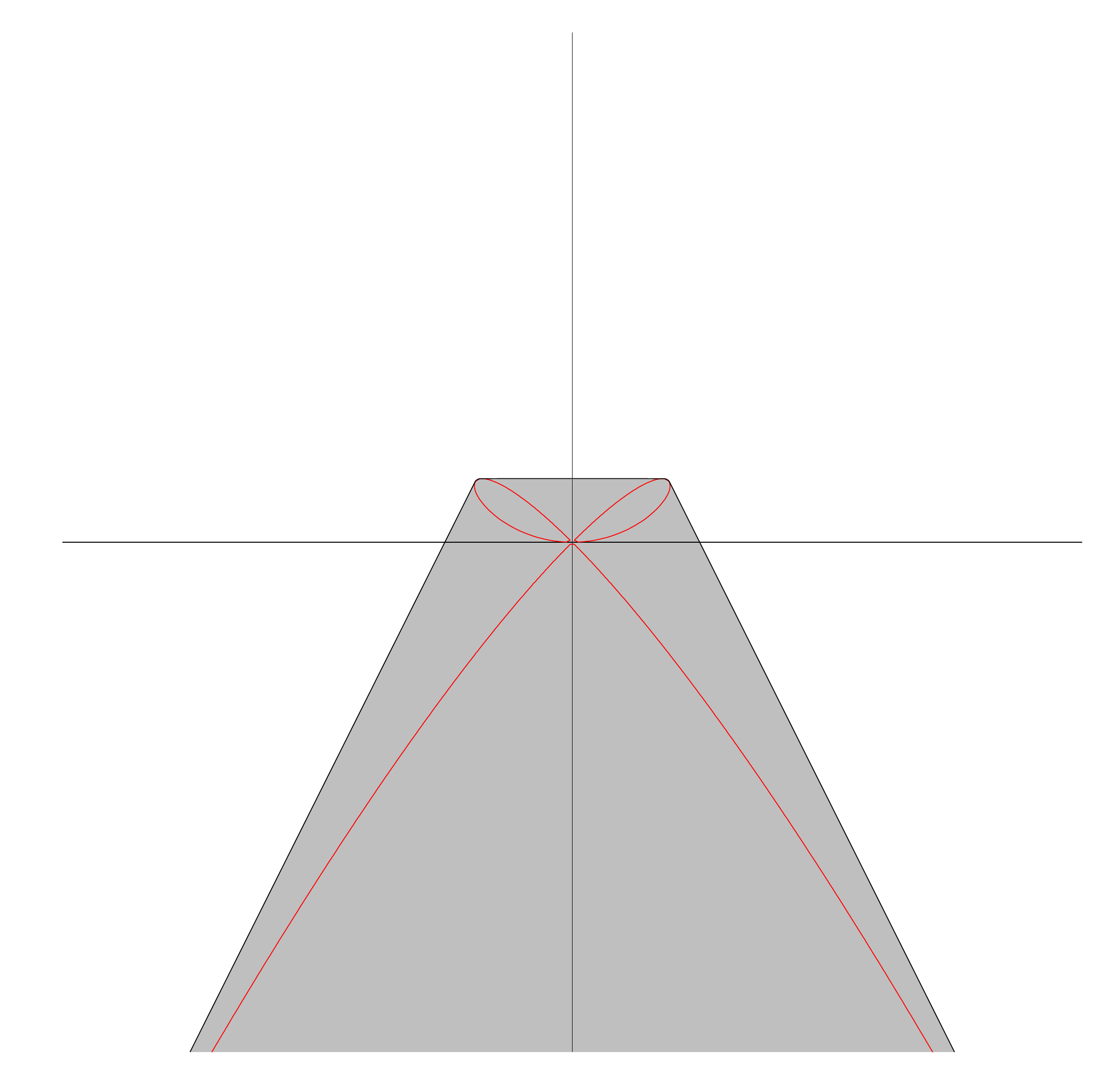}
\caption{The bow curve $S$ (red curve) and the third order spectrahedron
$\Omega_3(\wt{G})$ (shaded) in Exmaple
\ref{ex::notclosed2}.}\label{fig::notclosed2}
\end{figure}
\end{example}

\subsection{Pointedness of $\co{\cl{\wt{S}^{\text{o}}}}$}
When $\co{\cl{\wt{S}^{\text{o}}}}$ is not pointed, we divide   $S$ into $2^n$
parts along each axis. Let
\begin{equation}\label{mathE}
\mathcal{E}:=\{e=(e_1,\ldots,e_n)\mid e_i\in\{0, 1\},\ \ i=1,\ldots,n\}
\end{equation}
and for each $e\in\mathcal{E}$,
\begin{equation}\label{eq::Se}
S_e:=\{x\in\RR^n\mid g_i(x)\ge 0,\ (-1)^{e_j}x_j\ge 0,\ \ i=1,\ldots,m,\
j=1,\ldots, n\}.
\end{equation}
Then $S=\bigcup_{e\in\mathcal{E}}S_e$ and $\vert\mathcal{E}\vert=2^n$. For each
$e\in\mathcal{E}$, define $\wt{S}_e^\text{o}$, $\wt{S}_e^\text{c}$,  $\wt{S}_e$
and $\wt{G}_e$ as in $(\ref{eq::Ss})$ and (\ref{eq::defG}).
By Theorem \ref{th::pointedcone}, both of $\co{\cl{\wt{S}_e^\text{o}}}$
and $\co{\wt{S}_e^\text{c}}$ are closed and pointed for each $e\in\mathcal{E}$.

\begin{theorem}\label{th::notpointed}
Let $S\in\RR^n$ be a semialgebraic set defined as in $(\ref{eq::Snew})$.
Assume that
\begin{enumerate}[1.]
\item $S$ is closed at $\infty$;
\item
For each $e\in\mathcal{E}$, PP-BDR property holds for $\wt{S}_e$.
\end{enumerate}
Then $\cl{\co{S}}$ is the closure of a projected spectrahedron.
\end{theorem}
\begin{proof}
Fix an integer $k'$ such that  PP-BDR property holds for each $\wt{S}_e$ with
order $k'$. Note that $S_e$ may not be closed at $\infty$ for some
$e\in\mathcal{E}$. However, we show that
\begin{equation}\label{eq::inclusion}
\cl{\co{S_e}}\subseteq\cl{\Omega_{k'}(\wt{G}_e)}\subseteq\cl{\co{S}},\quad
\forall e\in\mathcal{E}.
\end{equation}
By Theorem \ref{th::omegatheta}, we get
\[
\cl{\co{S_e}}\subseteq\cl{\Omega_{k'}(\wt{G}_e)}\subseteq\wt{\tb}_{k'}(\wt{G}_e),\quad
\forall e\in\mathcal{E}.
\]
Fix a vector $u\not\in\cl{\co{S}}$. According to (\ref{eq::nn}), there exists a
polynomial $\td{f}\in\P[\wt{X}]_1$ such that
$\td{f}(1,u)<0$ and $\td{f}(\td{x})\ge 0$ on $\co{\cl{\wt{S}^\text{o}}}$. Since
$\cl{\wt{S}^\text{o}}=\wt{S}^\text{c}$ and
$\wt{S}_e^\text{c}\subseteq\wt{S}^\text{c}$ for each $e\in\mathcal{E}$, we have
$\td{f}(\td{x})\ge 0$ on each $\co{\wt{S}_e^\text{c}}$.
Because $\co{\wt{S}_e^{\text{c}}}$ is closed and pointed, by Theorem
\ref{th::pointedcone}, there exists a polynomial $\td{g}\in\P[\wt{X}]_1$ such
that $\td{g}(\td{x})>0$ on $\co{\wt{S}_e^{\text{o}}}$. We choose a
small $\epsilon>0$ such that $(\td{f}+\epsilon\td{g})(1,u)<0$ and
rename $\td{f}+\epsilon\td{g}$ as $\td{f}$, then $\td{f}(\td{x})>0$ on
$\wt{S}_e^{\text{c}}$. In particular, $\td{f}(\td{x})>0$ on $\wt{S}_e$.
Since $\wt{S}_e$ satisfies PP-BDR property with order $k'$,
we have $\td{f}\in\mathcal{Q}_{k'}(\wt{G})$ and $u\not\in\tbg{k'}$ due to the
fact that $\td{f}(1,u)<0$. It implies
$\wt{\tb}_{k'}(\wt{G}_e)\subseteq\cl{\co{S}}$ and
 $(\ref{eq::inclusion})$.
Therefore, we have
\[
\begin{aligned}
\cl{\co{S}}&=\cl{\co{\bigcup_{e\in\mathcal{E}}S_e}}\\
&=\cl{\co{\bigcup_{e\in\mathcal{E}}\cl{\co{S_e}}}}\\
&\subseteq\cl{\co{\bigcup_{e\in\mathcal{E}}\cl{\Omega_{k'}(\wt{G}_e)}}}\\
&\subseteq\cl{\co{\cl{\co{S}}}}\\
&=\cl{\co{S}},
\end{aligned}
\]
which implies
\[
\cl{\co{S}}=\cl{\co{\bigcup_{e\in\mathcal{E}}\cl{\Omega_{k'}(\wt{G}_e)}}}
=\cl{\co{\bigcup_{e\in\mathcal{E}}\Omega_{k'}(\wt{G}_e)}}.
\]
Since each $\Omega_{k'}(\wt{G}_e)$ is a projected spectrahedron, by
\cite[Theorem 2.2]{HeltonNie}, we have
\[
\cl{\co{\bigcup_{e\in\mathcal{E}}\Omega_{k'}(\wt{G}_e)}}=
\cl{\begin{array}{c|c} \sum_e\lambda_ex^{(e)} &\sum_e\lambda_e=1,\
\lambda_e\ge 0,\ x^{(e)}\in\Omega_{k'}(\wt{G}_e)
\end{array}}
\]
which is the closure of a projected spectrahedron.
\end{proof}


\begin{remark}\label{re::divide}
If $\mathcal{E}'$ is a subset of $\mathcal{E}$ such that
$S=\bigcup_{e\in\mathcal{E}'}S_e$, then according to the above proof, the
conclusion of Theorem \ref{th::notpointed} still holds if we replace
$\mathcal{E}$ by $\mathcal{E}'$.
\end{remark}

\paragraph{\bf Example \ref{ex::notpointed} (continued)}
For $S=\{(x_1,x_2)\in\RR^2\mid x_2^3-x_1^2\ge 0\}$, we have shown that
$\co{\cl{\wt{S}^{\text{o}}}}$ is not pointed, and the
modified theta bodies $(\ref{eq::newtheta})$ and Lasserre's relaxations
$(\ref{eq::omega})$ do not converge to $\co{\cl{\wt{S}}}$.
Due to Remark \ref{re::divide}, divide $S$ into two parts
\[
\begin{aligned}
S_{(0,0)}&:=\{(x_1,x_2)\in\RR^2\mid x_2^3-x_1^2\ge 0,\ x_1\ge 0,\ x_2\ge 0\},\\
S_{(1,0)}&:=\{(x_1,x_2)\in\RR^2\mid x_2^3-x_1^2\ge 0,\ -x_1\ge 0,\ x_2\ge 0\}.
\end{aligned}
\]
It is easy to check that PP-BDR property holds for $\wt{S}_{(0,0)}$ and
$\wt{S}_{(1,0)}$ with order one. Thus for any $k'\ge 1$, we have
\[
\begin{aligned}
&\Omega_{k'}\left(\wt{G}_{(0,0)}\right)=\{(x_1,x_2)\in\RR^2\mid  x_1\ge 0,\
x_2\ge 0\},\\
&\Omega_{k'}\left(\wt{G}_{(1,0)}\right)=\{(x_1,x_2)\in\RR^2\mid  x_1\le 0,\ x_2\ge 0\}.
\end{aligned}
\]
Clearly,
$\cl{\co{\Omega_{k'}(\wt{G}_{(0,0)})\cup\Omega_{k'}(\wt{G}_{(1,0)})}}=\cl{\co{S}}$
for any integer $k'\ge 1$. $\hfill\square$

However, if PP-BDR property does not hold with order $k'$ for some $\wt{S}_e$,
according to the proof of Theorem \ref{th::notpointed},
$\cl{\Omega_{k'}(\wt{G}_e)}$ may not be a subset of $\cl{\co{S}}$ for some
$e\in\mathcal{E}$. In  this case,
$\cl{\co{\bigcup_{e\in\mathcal{E}}\Omega_{k'}(\wt{G}_e)}}$ may contain
$\cl{\co{S}}$ strictly.

\begin{example}\label{ex::notpointed2}
Rotating the semialgebraic set $S$ in Example \ref{ex::notpointed} about the
origin $45^\circ$ counter-clockwise,  we get
\[
S':=\{(x_1,x_2)\in\RR^2\mid -\sqrt{2}(x_1-x_2)^3-2(x_1+x_2)^2\ge 0\},
\]
which is the left part of $\RR^2$ divided by the red curve in Figure
\ref{fig::notpointed}. Then, $\cl{\co{S}}$ is the closed half
space of $\RR^2$ partitioned by the line $X_2=X_1$.

By Remark \ref{re::divide}, set $\mathcal{E}'=\{(0,0),(1,0),(1,1)\}$ and divide
$S'=\bigcup_{e\in\mathcal{E}'}S'_{e}$ defined as in $(\ref{eq::Se})$.  Clearly,
for any integer $k'\ge 1$,
we have
\[
\Omega_{k'}\left(\wt{G}'_{(1,0)}\right)=\{(x_1,x_2)\in\RR^2\mid  x_1\le 0,\
x_2\ge 0\}.
\]
The third order spectrahedral approximation $\Omega_{3}\left(\wt{G}'_{(0,0)}\right)$
of $S'_{(0,0)}$ is shown shaded in Figure \ref{fig::notpointed}. As we can see,
the support line $X_2=X_1$ is approximated by $X_2=aX_1$ with $a<1$. The same
thing happens in the third quadrant. Numerically, we deduce
$\cl{\co{\bigcup_{e\in\mathcal{E}'}\Omega_3(\wt{G}'_e)}}=\RR^2$ which contains
$\cl{\co{S}}$ strictly.  $\hfill\square$
\begin{figure}
\includegraphics[width=0.35\textwidth]{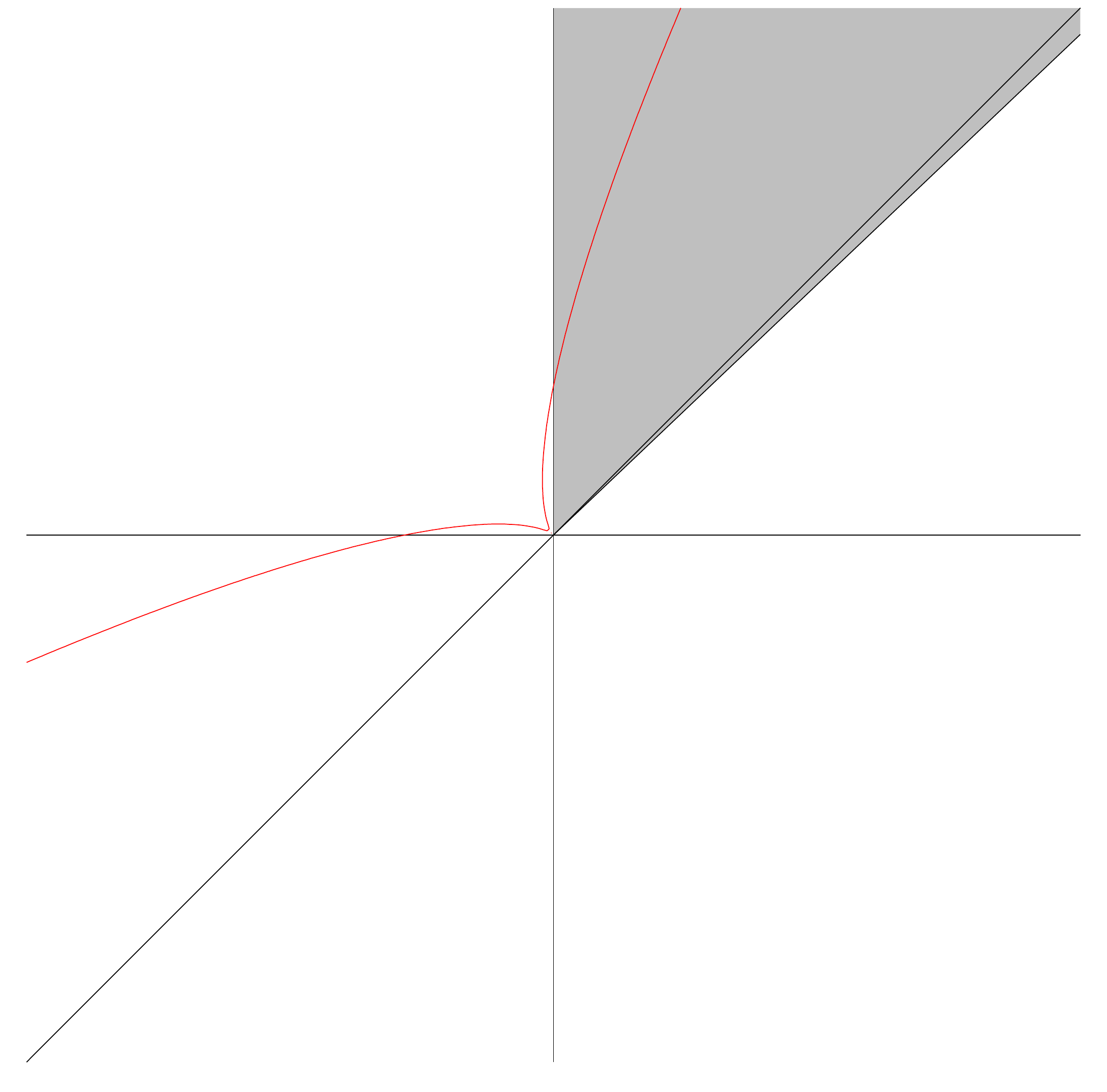}
\caption{The third order spectrahedral approximation
$\Omega_{3}\left(\wt{G}'_{(0,0)}\right)$ (shaded)
of $S'_{(0,0)}$ in Example \ref{ex::notpointed2}.}\label{fig::notpointed}
\end{figure}
\end{example}
Therefore, when  $\co{\cl{\wt{S}^{\text{o}}}}$ is not pointed, it becomes much more
complicate to approximate $\cl{\co{S}}$ properly  if PP-BDR property does not
hold on $\wt{S}_e$
for some $e \in \mathcal{E}$. We
leave this case for future investigations.

\vspace{15pt}
\noindent{\bfseries Acknowledgments:}
Feng Guo is supported by the Fundamental Research Funds for the Central
Universities. Lihong Zhi and Chu Wang  are  supported by NKBRPC 2011CB302400, the Chinese
    National Natural Science Foundation under Grants 91118001,
    60821002/F02.

\def\refname{\Large\bfseries References}
\bibliographystyle{plain}

\appendix
\section{}

In the following, we give the
proof of Proposition \ref{prop::pc}.

\begin{proof} [Proof of  Proposition \ref{prop::pc}]
We prove the properties are equivalent by showing the implications  \cite{CANO}
\[
(a) \Rightarrow (b) \Rightarrow (c) \Rightarrow (d) \Rightarrow (e)
\Rightarrow (f) \Rightarrow (a).
\]

\begin{itemize}
\item
(a)$\Rightarrow$(b).
Since $K$ is pointed, by (\ref{eq::conedual}), we have
\[
\cl{K^*-K^*}=(K\cap-K)^*={0}^*=\RR^n.
\]

\item
(b)$\Rightarrow$(c). It follows from
the fact that $K^*-K^*$ is a subspace in $\RR^n$ and therefore it
 is closed.
\item
(c)$\Rightarrow$(d).
Since $K^*$ is always nonempty, by \cite[Theorem 6.2]{convexanalysis}, it
 has nonempty relative
interior. By \cite[Theorem 2.7]{convexanalysis} and (c), we have
$\aff{K^*}=K^*-K^*=\RR^n$ and thus $K^*$ has
 nonempty interior.

\item
(d)$\Rightarrow$(e).
 Let $y$ be an interior of $K^*$, then $\langle
y,\ x\rangle>0$ for all $0\neq x\in K$. Let $\epsilon:=\min\{\langle
y,\ u\rangle\mid u\in K,\ \Vert u\Vert_2=1\}$, then $\epsilon>0$ and
$\langle y,x\rangle\geq \epsilon \|x\|$ for all
$x \in K$.

\item
(e)$\Rightarrow$(f).
 By (e), there
exist a vector $y$  and real $\epsilon>0$ satisfying $1/\Vert
y\Vert_2\le\Vert x\Vert_2\le 1/\epsilon$ for all $x\in  C:=\{x\in
K\mid\langle x,\ y\rangle=1\}$, thus $C$ is bounded and
$0\not\in\cl{C}$.  For every $0\neq u\in K$, we have $\langle u,\
y\rangle>0$ and $u/\langle u,\ y\rangle\in C$, i.e., $C$ is a
bounded
  base of $K$.

\item
(f)$\Rightarrow$(a).
 Suppose $C$ is a bounded base of $K$. If there exists $0\neq x\in K\cap -K$,
then $c_1x\in C$ and $-c_2x\in C$ for some $c_1, c_2 \in \RR_+$. Since
$C$ is a convex set, we have $0\in C$.  This  contradicts to the
assumption that $0\notin\cl{C}$.
\end{itemize}
\end{proof}

\end{document}